\documentclass[11pt]{amsart}
\usepackage{graphicx}
\usepackage{amsmath,amsthm,amssymb,enumerate}
\usepackage{euscript,mathrsfs}
\usepackage[citecolor=blue,colorlinks=true]{hyperref}
\usepackage[left=2.5cm,right=2.5cm,top=4.5cm,bottom=4.5cm]{geometry}
\usepackage{color}
\catcode`\@=11 \@addtoreset{equation}{section}

\catcode`\@=12

\allowdisplaybreaks

\newtheorem{Theorem}{Theorem}[section]
\newtheorem{Proposition}[Theorem]{Proposition}
\newtheorem{Lemma}[Theorem]{Lemma}
\newtheorem{Corollary}[Theorem]{Corollary}

\theoremstyle{definition}
\newtheorem{Definition}[Theorem]{Definition}

\newtheorem{Remark}[Theorem]{Remark}

\newcommand{\bTheorem}[1]{
\begin{Theorem} \label{T#1} }
\newcommand{\eT}{\end{Theorem}}

\newcommand{\bProposition}[1]{
\begin{Proposition} \label{P#1}}
\newcommand{\eP}{\end{Proposition}}

\newcommand{\bLemma}[1]{
\begin{Lemma} \label{L#1} }
\newcommand{\eL}{\end{Lemma}}

\newcommand{\bCorollary}[1]{
\begin{Corollary} \label{C#1} }
\newcommand{\eC}{\end{Corollary}}

\newcommand{\bRemark}[1]{
\begin{Remark} \label{R#1} }

\newcommand{\eR}{\end{Remark}}

\newcommand{\bDefinition}[1]{
\begin{Definition} \label{D#1} }
\newcommand{\eD}{\end{Definition}}

\newcommand{\Exp}{\mathbb{E}}

\newcommand{\Ds}{\mathbb{D}_x}

\newcommand{\pig}{P_{\Gamma_{\rm out}}}

\DeclareMathOperator{\supp}{supp}

\newcommand{\bfg}{\mathbf{g}}

\newcommand{\tvm}{\tilde{\vc{m}}}
\newcommand{\bfphi}{\boldsymbol{\varphi}}

\newcommand{\bFormula}[1]{
\begin{equation} \label{#1}}
\newcommand{\eF}{\end{equation}}

\newcommand{\Ov}[1]{\overline{#1}}

\newcommand{\aleq}{\lesssim}

\newcommand{\vr}{\varrho}

\newcommand{\tvr}{\tilde \vr}

\newcommand{\vu}{\vc{u}}
\newcommand{\vm}{\vc{m}}

\newcommand{\vc}[1]{{\bf #1}}

\newcommand{\Div}{{\rm div}_x}
\newcommand{\Grad}{\nabla_x}

\newcommand{\dx}{\,{\rm d} {x}}

\newcommand{\dt}{\,{\rm d} t }

\newcommand{\intO}[1]{\int_{Q} #1 \ \dx}

\newcommand{\D}{{\rm d}}

\newcommand{\ep}{\varepsilon}

\newcommand{\tmop}{\text}
\def\softd{{\leavevmode\setbox1=\hbox{d}%
          \hbox to 1.05\wd1{d\kern-0.4ex{\char039}\hss}}}
\definecolor{Cgrey}{rgb}{0.85,0.85,0.85}
\definecolor{Cblue}{rgb}{0.50,0.85,0.85}
\definecolor{Cred}{rgb}{1,0,0}
\definecolor{fancy}{rgb}{0.10,0.85,0.10}

\allowdisplaybreaks

\newcommand\Cbox[2]{%
    \newbox\contentbox%
    \newbox\bkgdbox%
    \setbox\contentbox\hbox to \hsize{%
        \vtop{
            \kern\columnsep
            \hbox to \hsize{%
                \kern\columnsep%
                \advance\hsize by -2\columnsep%
                \setlength{\textwidth}{\hsize}%
                \vbox{
                    \parskip=\baselineskip
                    \parindent=0bp
                    #2
                }%
                \kern\columnsep%
            }%
            \kern\columnsep%
        }%
    }%
    \setbox\bkgdbox\vbox{
        \color{#1}
        \hrule width  \wd\contentbox %
               height \ht\contentbox %
               depth  \dp\contentbox
        \color{black}
    }%
    \wd\bkgdbox=0bp%
    \vbox{\hbox to \hsize{\box\bkgdbox\box\contentbox}}%
    \vskip\baselineskip%
}

\date{}

\usepackage{dsfont}

\newcommand{\tbf}{\textbf}

\newcommand{\mc}{\mathcal}
\newcommand{\mds}{\mathds}

\newcommand{\vrho}{\varrho}
\newcommand{\vphi}{\varphi}
\newcommand{\oline}{\overline}
\newcommand{\uline}{\underline}
\newcommand{\ra}{\rightarrow}

\newcommand{\g}{\gamma}

\renewcommand{\t}{\tau}

\newcommand{\N}{\mathbb{N}}

\renewcommand{\div}{{\rm div}\,}

\allowdisplaybreaks

\def\div{{\rm div}\,}

\begin{document}


\title{Ergodic theory for energetically open compressible fluid flows}

\author{Francesco Fanelli}
\address[F. Fanelli]{Univ. Lyon, Universit\'e Claude Bernard Lyon 1, CNRS UMR 5208, Institut Camille Jordan, F-69622 Villeurbanne, France}
\email{fanelli@math.univ-lyon1.fr}
\thanks{The work of F.F. has been partially supported by the LABEX MILYON (ANR-10-LABX-0070) of Universit\'e de Lyon, within the program ``Investissement d'Avenir''
(ANR-11-IDEX-0007), by the projects BORDS (ANR-16-CE40-0027-01), SingFlows (ANR-18-CE40-0027) and CRISIS (ANR-20-CE40-0020-01), all operated by the French National Research Agency (ANR)}

\author{Eduard Feireisl}
\address[E.Feireisl]{Institute of Mathematics AS CR, \v{Z}itn\'a 25, 115 67 Praha 1, Czech Republic
\and Institute of Mathematics, TU Berlin, Strasse des 17.Juni, Berlin, Germany }
\email{feireisl@math.cas.cz} 
\thanks{The research of E.F. leading to these results has received funding from
the Czech Sciences Foundation (GA\v CR), Grant Agreement
18--05974S. The Institute of Mathematics of the Academy of Sciences of
the Czech Republic is supported by RVO:67985840.}

\author{Martina Hofmanov\'a}
\address[M. Hofmanov\'a]{Fakult\"at f\"ur Mathematik, Universit\"at Bielefeld, D-33501 Bielefeld, Germany}
\email{hofmanova@math.uni-bielefeld.de}
\thanks{M.H. gratefully acknowledges the financial support by the German Science Foundation DFG via the Collaborative Research Center SFB 1283. This project has received funding from the European Research Council (ERC) under the European Union's Horizon 2020 research and innovation programme (grant agreement No. 949981). }

\begin{abstract}

The ergodic hypothesis is examined for energetically open fluid systems represented by the barotropic Navier--Stokes equations with general inflow/outflow boundary conditions.
We show that any globally bounded trajectory generates a stationary statistical solution, 
which is interpreted as a stochastic process with continuous trajectories supported by the family of weak solutions
of the problem. The abstract Birkhoff--Khinchin theorem is applied to obtain convergence (in expectation and a.s.) of ergodic averages for any bounded Borel measurable function of state variables associated to any stationary solution.
Finally, we show that validity of the ergodic hypothesis is determined by the behavior of entire solutions
(i.e. a solution defined for any $t\in R$). In particular, the ergodic averages converge for \emph{any} trajectory 
provided its $\omega-$limit set in the trajectory space supports a unique (in law) stationary solution. 

\end{abstract}

\keywords{barotropic Navier--Stokes system, ergodic theory, inflow/outflow boundary conditions, 
stationary statistical  solution}

\date{\today}

\maketitle

\tableofcontents


\section{Introduction}
\label{i}

One of the central open problems in statistical theories of hydrodynamic turbulence is the validity of the so-called \emph{ergodic hypothesis},
typically taken for granted by physicists and engineers, see e.g. \cite[Section 2.1]{Bat_Turb}. Roughly speaking, the ergodic hypothesis states the following:

\begin{quotation}
\emph{Time averages along trajectories of the flow converge, for large enough times, to an ensemble average given by a certain probability measure.}
\end{quotation}
Such a measure, if it exists, can be shown to be invariant for the dynamics. In other words, the ergodic hypothesis postulates that, over the long haul,
the system approaches its statistical equilibrium and the statistics of fully developed turbulence can be  
completely described by means of a single invariant measure. More detailed information on the mathematical theory of fluid turbulence 
can be found in the survey paper by Constantin \cite{Const1}, Bercovici et al. \cite{BeCoFoMa}, Constantin and 
Procaccia \cite{CoPro}, Foia\c{s} et al. \cite{FMRT}, and in the references cited therein. 

Another fundamental principle is the Second law of thermodynamics stated in its most flagrant form by Clausius: 

\begin{quotation}

\emph{
The energy of the world is constant; its entropy tends to a maximum.}

\end{quotation}

\noindent
Translating ``world'' as \emph{energetically closed system} we conclude that its dynamics must approach the state with 
maximal entropy that is necessarily an \emph{equilibrium}. Note that this is perfectly compatible with the ergodic hypothesis 
with the associated measure the Dirac mass sitting on the equilibrium state. Accordingly, the genuine turbulence can persist in the long 
run only for \emph{energetically open} systems, where the coercive effect of dissipation is counterbalanced by the action 
of external forces. Indeed, for \emph{energetically closed} systems of real (dissipative) fluids, the entropy (energy) plays a role of a Lyapunov 
function and the dynamics converge towards an equilibrium as predicted by the Second law, see e.g. the discussion in \cite{FP20}. To avoid such a boring scenario, a non--conservative \emph{volume force} is usually introduced in the literature, either deterministic or even of stochastic type 
in the momentum equation. Although this is convenient from the mathematics point of view, the natural volume force acting in the real world applications is a potential (conservative) gravity field. A truly open real fluid must interact with the outer world through 
the \emph{boundary} of the physical space.

Transforming the physically obvious statements in the language of rigorous mathematical theorems is hampered by the fact that the associated mathematical models of fluids are still relatively poorly understood. The iconic example is the incompressible 
Navier--Stokes system in the relevant 3D geometry, for which the fundamental questions of uniqueness and even existence of (smooth) 
solutions are largely open, see Fefferman \cite{Feff}. Still the undeniable success of the model in numerical simulations of the real world problems indicates its sound physical background. Note that well--posedness of the model in the mathematical sense is not 
fundamental when addressing the properties of the system in the long run. The only necessary piece of information is the 
existence of \emph{global--in--time solutions} together with a kind of \emph{energy/entropy balance equation} encoding the Second law.   
Both are available in the framework of weak solutions used in the present paper.  

Our goal is to investigate the ergodic structure and validity of the ergodic hypothesis in the context of \emph{energetically open fluid systems} modeled by the barotropic Navier--Stokes equations with general outflow/inflow boundary conditions.
To the best of our knowledge, this is the first work addressing this question in such a physically relevant situation.
The system 
describes the time evolution of the mass density $\vr = \vr(t,x)$ and the bulk velocity $\vu = \vu(t,x)$ of a general compressible 
viscous fluid confined to a physical domain $Q \subset R^3$: 
\begin{equation} \label{i1}
\begin{split}
\partial_t \vr + \Div (\vr \vu) &= 0,\\
\partial_t (\vr \vu) + \Div (\vr \vu \otimes \vu) + \Grad p(\vr) &= 
\Div \mathbb{S} + \vr \vc{g} ,\ (t,x) \in (0,T) \times Q,
\end{split}
\end{equation}
where $\mathbb{S}$ is the viscous stress tensor given by \emph{Newton's rheological law}
\begin{equation} \label{i2}
\mathbb{S}(\Ds \vu) = \mu \left( \Grad \vu + \Grad^t \vu - \frac{2}{3} \Div \vu \mathbb{I} \right) + 
\lambda \Div \vu \mathbb{I},\ \Ds \vu \equiv \frac{1}{2} \Big( \Grad \vu + \Grad^t \vu \Big),
\end{equation}
and $\vc{g}$ a driving force. The system communicates with the outer world through the inflow/outflow boundary conditions:  
\begin{align}
\vu|_{\partial Q} = \vu_b,& \ \vu_b = \vu_b(x), \ x \in \partial Q,  \label{i3} \\
\vr|_{\Gamma_{\rm in}} = \vr_b, \ \vr_b = \vr_b(x), &\ x \in \Gamma_{\rm in}, \ 
\Gamma_{\rm in} \equiv \left\{ x \in \partial Q \ \Big|\ \vu_b \cdot \vc{n} < 0 \right\}. \label{i4}
\end{align}
Here and in the following, we assume that $\vu_b\in C^1(\Ov{Q};R^3)$ is a function defined on the whole spatial domain $\Ov{Q}$.

The system \eqref{i1} can be seen as an evolutionary problem in terms of the \emph{conservative variables} $[\vr, \vm]$, where 
$\vm = \vr \vu$ is the linear momentum. Unlike its incompressible counterpart, the barotropic Navier--Stokes system 
captures all essential features of problems in continuum fluid mechanics, notably the ``hyperbolic'' character of the mass 
transport combined with the ``parabolic'' structure of the momentum equation. 
 
\medbreak

As the problem of well--posedness (in terms of the initial data) of \eqref{i1}--\eqref{i4} is still largely open,
we put forward a different point of view on the dynamical properties, which
goes back to Sell \cite{SEL} and M\'alek, Ne\v cas \cite{MANE} in the context of the attractor theory, and It\^o, Nisio \cite{Itno} in the framework of SDE. The leading idea is to consider 
the whole trajectory as the ``initial datum'', whereas the dynamics is represented by simple time shifts: 
\[
S_\tau [\vr, \vm] (t, \cdot) \equiv [\vr(t + \tau, \cdot), \ \vm(t + \tau, \cdot)], \tau \geq 0 \ (\tau \in R).
\]
Note that $\tau < 0$ is relevant for \emph{entire solutions} defined for 
all $t \in R$.

In fact, a similar strategy has recently been implemented by Foia\c{s}, Rosa and Temam \cite{FoRoTe2, FoRoTe1} in studying convergence of time averages for
the incompressible Navier--Stokes system driven by a non--conservative volume force. 
The investigation of both those works strongly relied on the notion of \emph{stationary statistical solution}, which is closely related to the concept of
\emph{invariant measure} for well--posed problems.

\emph{Statistical solutions} to problems in fluid mechanics have been introduced in the pioneering works of Foia\c{s} \cite{Foias} and 
Vishik, Fursikov \cite{VisFur} again in the context of incompressible fluid models. Roughly speaking, they are 
measures supported by the solution trajectories of the underlying system of PDE's. See \cite{ConWu} and \cite{FoRoTe3} for related studies and additional details.
We also refer to Jiang and Zhao \cite{Jiang-Zhao}, Zhao et al. \cite{ZCLuk, ZLiC, ZLiLuk} for recent studies based on the notion of statistical solutions for different systems of PDEs.
When the measure {is invariant with respect to (positive) time shifts}, the statistical solution is said stationary.

More recently, a new class of (dynamical) statistical solutions 
for the Navier--Stokes system \eqref{i1}--\eqref{i4} was introduced in \cite{FanFei}. The statistical solutions are {identified with the pushforward measure} associated to a semiflow selection among all possible (weak) solutions of the problem emanating from fixed initial data. {Inspired by \cite{FanFei}, 
we advocate the alternative interpretation of a statistical solution} as a stochastic process with continuous paths {supported by solutions of the problem.}
From this perspective, statistical solutions to \emph{deterministic} problems can be seen as a special case of solutions of the related stochastic PDE
(SPDE in brief) with a stochastic forcing term and random data, where the stochastic forcing vanishes.
Relevant global--in--time existence results for statistical solutions can be therefore deduced from the SPDE theory, see e.g. Flandoli 
\cite{Fland} for the incompressible case, and \cite{BrFeHobook, BrHo} for the compressible case (note, however, that those results concern energetically closed systems only).

In accordance with the above delineated general strategy, 
a \emph{stationary statistical solution} is simply a \emph{stationary} stochastic process supported by global--in--time (weak) solutions of the problem. More precisely, the process is supported by entire solutions (defined for all $t \in R$) and invariant with respect to the action of the time shift operators $S_\t$ defined above.
The \emph{probability law} of this process in the trajectory space is the stationary statistical solution in the sense of Vishik and Fursikov \cite{VisFur} (see also \cite{FoRoTe1}). The standard Krylov--Bogolyubov theory gives rise to a particular class of 
stationary solutions generated by ``ergodic envelopes'' of solutions $[\vr,\vm]$ with globally bounded energy. The construction 
relies on tightness of 
the family of probability measures
\begin{equation} \label{ERGOD}
\nu_T \equiv \frac{1}{T} \int_0^T \delta_{S_\tau [\vr, \vm]} \D \tau
\end{equation} 
and continuity of the shift operators in the trajectory space. Any accumulation point $\nu$ for $T_n \to \infty$ of 
$\nu_{T_n}$ in the space of measures (sitting on the trajectory space) represents a law of a stationary statistical solution.
It can be shown that any such measure (stationary solution) is supported on the $\omega$--limit set
$\omega[\vr, \vm]$
of $[\vr, \vm]$ in the 
space of trajectories. As the $\omega$--limit set can be interpreted as the largest region in the phase space 
visited by the solution $[\vr, \vm]$ 
in the long run, the ergodic hypothesis can be reformulated in the statement: 

\begin{quotation}

\emph{The measure $\nu$ is ergodic with respect to the shift transformation on $\omega[\vr, \vm]$.}

\end{quotation}

\noindent Here, \emph{ergodic} means that for any shift invariant Borel subset $B$ of $\omega[\vr, \vm]$ either 
$\nu (B) = 1$ or $\nu(B) = 0$. Although the above definition is mathematically acceptable, its formulation is a bit awkward 
referring to the concept of trajectory space, time shifts, $\omega$--limit set etc. We therefore adopt a weaker form.
To this end, we need to specify the \emph{phase space} associated to the Navier--Stokes system, meaning the function space $H$ in which the solution 
$[\vr, \vm]$ lives at any time instant, 
\[
[\vr(t, \cdot), \vm(t, \cdot)] \in H \ \mbox{for any}\ t .
\]
As the \emph{energy} associated to the Navier--Stokes system reads 
\begin{equation*}
\begin{split}
E \left(\vr, \vm \ \Big| \vu_b \right) \equiv \left[ \frac{1}{2} \vr |\vu - \vu_b|^2 + P(\vr) \right] &= 
\left[ \frac{1}{2} \frac{|\vm|^2}{\vr} - \vm \cdot \vu_b + \frac{1}{2} \vr |\vu_b|^2 + P(\vr) \right],  \\
\vm \equiv \vr \vu,\ \ P'(\vr) \vr  - P(\vr) &= p(\vr),
\end{split}
\end{equation*} 
a natural choice is to identify the phase space $H$ through suitable ``energy'' norm. However, in order to encode fundamental compactness properties of the solutions, throughout this paper we set
\[
H = L^1(Q) \times W^{-k,2}(Q; R^3).
\] 
We can now state our working definition of (pointwise) ergodic hypothesis for the Navier--Stokes system \eqref{i1}--\eqref{i4}:

\begin{quotation}

\emph{Navier--Stokes system \eqref{i1}--\eqref{i4} complies with the ergodic hypothesis if for any solution $[\vr, \vm]$ defined 
on a time interval $(\tau, \infty)$, $\tau \geq - \infty$, the limit} 
\begin{equation} \label{I5a}
\lim_{T \to \infty} \frac{1}{T} \int_0^T F(\vr(t, \cdot), \vm(t, \cdot) ) \ \dt 
\end{equation}
\emph{exists for any bounded continuous functional $F$ on $H$.}

\end{quotation}

Loosely speaking, we may contrast the ergodic hypothesis for energetically open and closed system as follows: 
\begin{itemize}

\item {\bf Closed system.} 

For any global--in--time solution $[\vr, \vm]$, 
the total energy $\intO{ E(\vr, \vm | \vu_b ) }$ is a Lyapunov function, and 
\[
[\vr, \vm](t, \cdot) \to [\tvr , \tvm]  \ \mbox{as}\ t \to \infty,   
\]
where $[\tvr , \tvm]$ is an equilibrium state.

\item {\bf Open system.}

For any global--in--time solution $[\vr, \vm]$,
the total energy $\intO{ E(\vr, \vm | \vu_b ) }$ remains bounded as $t \to \infty$, and 
\[
\frac{1}{T} \int_0^T \delta_{[\vr, \vm](t, \cdot)} \ \dt \to \mathcal{\mathcal{V}} \ \mbox{as} \ T \to \infty,\ 
\mbox{where}\ \mathcal{V}  \ \mbox{is a probability measure on}\ H .
\]

\end{itemize}

\medbreak

Our strategy can be delineated as follows. Having identified a suitable phase space $H$, the trajectory space 
$\mathcal{T}$ is simply defined as the space $\mathcal{T} = C(R;H)$ of continuous functions ranging in $H$. Note that 
solutions defined on $(\tau, \infty)$, $\tau$ finite, can be identified with trajectories in $\mathcal{T}$ by extending them 
to be constant for $t \leq \tau$. Adopting this convention, we may define the $\omega$--limit set associated to 
a solution $[\vr, \vm]$ as 
\[
\omega [\vr, \vm] = \left\{ [r, \vc{w}] \in \mathcal{T} \ \Big|\ [\vr, \vm] (\cdot + T_n, \cdot) \to [r, \vc{w}] 
\ \mbox{in}\ \mathcal{T} \ \mbox{for some}\ T_n \to \infty \right\}.
\]
Next, we identify the class of solutions relevant for the long--time dynamics, 
\[
\mc U(\oline E) = \left\{ [\vr, \vm] \ \Big| \ [\vr,\vm] 
- \mbox{entire solution},\ \intO{ E\left(\vr, \vm \ \Big| \ \vu_b \right) (t, \cdot) } \leq \Ov{E} \ \mbox{for all}\ t \in R \right\} .
\]

Our first crucial observation is contained in the following result (see Theorem~\ref{aP1} for the precise statement).

\begin{Theorem}\label{thm:In1}
Let $[\vr, \vm]$ be a global--in--time solution of the Navier--Stokes system such that 
\[
\limsup_{t \to \infty} \intO{ E\left(\vr, \vm \ \Big| \ \vu_b \right) (t, \cdot) } \leq \Ov{E}.
\]

Then 
\begin{itemize}
\item 
$
\emptyset \ne \omega[\vr, \vm] \subset \mc U(\oline E)
$;
\item 
the set $\mc U(\oline E)$ is compact in $\mathcal{T}$.

\end{itemize}
\end{Theorem}

The above result can be rephrased in terms of \emph{sequential compactness} of the time shifts of the solution 
$[\vr, \vm]$, where actually the choice of the 
space $H$ plays a marginal role. The crucial observation is that the bounds provided by the energy are sufficient to show that functions 
in the $\omega$--limit set are again \emph{solutions} of the nonlinear problem \eqref{i1}--\eqref{i4}. This property is quite 
standard in the class of purely ``parabolic'' problems, 
in particular the \emph{incompressible} Navier--Stokes system, 
where compactness is guaranteed by the presence of diffusion terms, and the problem boils down to a simple modification of the arguments of the \emph{existence} theory. As a matter of fact, the argument of compactness 
in the existence theory for the compressible Navier--Stokes system is rather delicate because of the density. Indeed the latter satisfies only the equation of continuity $($\ref{i1}$)_1$ that allows for propagation of oscillations. Compactness relies 
on an ingenious argument of Lions \cite{LI4} that combines: 
\begin{itemize}
\item compactness (strong convergence) of the sequence $\vr_n(0, \cdot)$ of the initial data; 
\item ``weak continuity'' of the effective viscous flux.  
\end{itemize}
As the information on compactness of the ``initial data'' is apparently missing for the time shifts of global--in--time solutions, 
a more sophisticated argument must be used based on qualitative estimates of the temporal decay rate of density oscillations.

Another interesting problem, largely open in the case of energetically open systems, is the existence 
of a bounded entire solution, namely if $\mathcal{U}(\Ov{E}) \ne \emptyset$ for some large $\Ov{E}$.
Some sufficient conditions guaranteeing this property are given in Section \ref{ss:existence}.
We refer also to \cite{Bre-F-Nov}, where existence of globally bounded trajectory is shown for the barotropic Navier-Stokes system \eqref{i1}--\eqref{i4}
supplemented with the so-called hard-sphere pressure equation of state.

Theorem \ref{thm:In1} is a cornerstone of our investigation, since it opens the road to all the subsequent analysis.
First of all, the existence of an entire solution with globally bounded energy, namely the fact that the set $\mc U(\oline E)$ is non--empty (for some $\oline E>0$),
is enough to guarantee the existence of a stationary statistical solution (see Theorem \ref{IT1} below for details).

\begin{Theorem}\label{thm:In2}
Let $[\vr, \vm]$ be solution with globally bounded energy. Then there exists a stationary statistical solution supported 
by the $\omega-$limit set $\omega[\vr, \vm] \subset \mathcal{U}(\Ov{E})$.
\end{Theorem}

The proof of Theorem \ref{thm:In2} is performed via the standard Krylov--Bogolyubov's argument 
based on tightness of the probability measures \eqref{ERGOD} defined on the {trajectory space} $\mathcal{T}$. 
Any limit point $\nu$ (in the narrow topology) for a suitable $T_n \to \infty$ in \eqref{ERGOD} 
represents the desired law of a stationary statistical  solution.
Note that this procedure allows us to avoid the use of generalized limits, employed in \cite{FoRoTe1}, and to work directly with classical limits.

Moreover, it turns out that, with Theorem~\ref{thm:In1} at hand, the theory of dynamical systems becomes accessible, even in the comfortable setting of dynamical systems with compact state space.
Specifically, Theorems \ref{thm:In1} and \ref{thm:In2} together permit us to reformulate the problem as a measure--preserving dynamical system generated by the action of time shifts on trajectories.
The huge machinery of dynamical systems thus applies (in a quite direct way) to our setting. In particular, 
we can show that the collection of all $\omega-$limit sets supports all invariant measures. In addition, as a by--product of Theorem~\ref{thm:In2}, we show the existence of a recurrent solution belonging
to its own $\omega-$limit set (see Corollary~\ref{cor:rec} in this respect).

In the context outlined above, a remarkable feature of stationary statistical solutions, expressed in terms of the celebrated Birkhoff--Khinchin theorem,
is the convergence of the ergodic averages (we refer to Theorem \ref{AT1} and its extension Theorem~\ref{ACC1} for the precise statements).

\begin{Theorem}\label{thm:IN4}
Let $[\vr,\vm]$ be a stationary statistical solution to the barotropic Navier--Stokes system, and let $F$ be  a bounded 
{Borel}
functional on the state space $H$. Then the limit
\begin{equation} \label{i71}
\lim_{T\to\infty}\frac{1}{T} \int_0^T F (\vr(t, \cdot), \vm(t, \cdot) ) \dt 
\end{equation}
exists a.s.
\end{Theorem}

At this point, it is tempting to immediately compare this result with the desired ergodic hypothesis. Nonetheless, two different issues show that it
does not give a satisfactory answer, yet.

On the one hand, the catch is in the fact that the limit exists only a.s.; in other words, Theorem~\ref{thm:IN4} entails 
{only trajectories, with the corresponding (initial) data, that belong to the support of the stationary solution (measure) 
$\nu$. Note that speaking about ``initial'' data in this context is irrelevant, as the choice of the ``initial'' time is irrelevant 
for the stationary solution. In addition, the set of the corresponding data can be rather small for certain stationary solutions. 
In particular, if the stationary solution is deterministically stationary, meaning the trajectory is independent of time, then 
the piece of information hidden in \eqref{i71} is trivial as the stationary measure reduces to a Dirac mass.
}

On the other hand, Theorem~\ref{thm:IN4} only gives a convergence to a random variable, and not to an \emph{ensemble average}. As it will be shown in Theorem~\ref{AC1}, the limit \eqref{i71} can be described by a \emph{conditional expectation} with respect to the $\sigma$--algebra of
shift invariant sets. Hence it becomes a true expectation exactly when this $\sigma$--algebra is trivial for the law of the stationary solution $[\vr,\vm]$, in other words when the law is \emph{ergodic}.
It turns out that, based on  Krein--Milman's theorem, we are able to prove existence of such ergodic stationary statistical solutions.
In addition, as a consequence of the so-called ergodic decomposition (see Theorem~\ref{thm:AC1}), we will show that the union of the supports of all invariant measures coincides with the union
of the supports of all ergodic invariant measures, and that different ergodic invariant measures must be singular.

Thus our conclusion for the validity of ergodic hypothesis can be summarized in the following statement.

\begin{Theorem}\label{thm:In5}

{Let $[\vr, \vm]$ be a global--in--time solution with uniformly bounded energy.}

Then the limit \eqref{I5a} exists if the 
$\omega$--limit set $\omega[\vr, \vm]$ supports a unique (ergodic) invariant  measure, i.e., a unique stationary statistical solution. 
Accordingly, the ergodic hypothesis, i.e., validity of \eqref{I5a} for any trajectory, holds, provided 
all trajectories admit a uniform energy bound, and 
all associated $\omega$--limit sets 
support a unique (ergodic) invariant measure.
\end{Theorem}

To conclude the introductory discussion, we remark that our result offers only a partial affirmative answer to the ergodic hypothesis.
On the other hand, Theorem \ref{thm:In5} points at a new direction of investigation in order to get a better understanding of the ergodic hypothesis, suggesting 
that the structure of the $\omega-$limit sets associated to entire solutions with globally bounded energy plays a crucial role in all this matter.

\section{Preliminaries, main hypotheses, weak solutions} 
\label{h} 

We start by stating the precise structural restrictions imposed on the constitutive equations and the data for the 
Navier--Stokes system \eqref{i1}--\eqref{i4}. To begin, we suppose that the boundary velocity $\vu_b$ admits an extension in 
$\Ov{Q}$. Moreover, we extend the total energy,   
\[
E \left(\vr, \vm \ \Big| \vu_b \right) = \left\{ \begin{array}{l}  
\left[ \frac{1}{2} \frac{|\vm|^2}{\vr} - \vm \cdot \vu_b + \frac{1}{2} \vr |\vu_b|^2 + P(\vr) \right] \ \mbox{if}\ \vr > 0, 
\\ \\ \infty \ \mbox{if}\ \vr < 0 \ \mbox{or}\ \vr = 0, \vm \ne 0, 
\\ \\ 0 \ \mbox{if}\ \vr - 0, \ \vm = 0,
\end{array} \right. 
\]
to a convex, l.s.c. function of variables $(\vr, \vm) \in R^4$.

\subsection{Boundary data}
We proceed by a list of basic hypotheses to be imposed on the data as well as the geometry of the spatial domain. They are probably not optimal but necessary for the available mathematical theory. We focus on the physically relevant 3D setting, similar results can be shown in the 2D case. We leave apart the 1D geometry, where better results are expected as the problem is known to be well posed with respect to the initial data.

We suppose that $Q \subseteq R^3$ is a bounded domain of class $C^2$, with a given vector field  $\vu_b \in C^1(\partial Q;R^{3})$. 
Accordingly, we decompose   
\begin{equation*}
\partial Q = \Gamma_{\rm in} \cup \Gamma_{\rm out}, \ 
\Gamma_{\rm in} = \left\{ x \in \partial Q \ \Big|\ \vu_b(x) \cdot \vc{n}(x) < 0 \right\},\ 
\Gamma_{\rm out} = \left\{ x \in \partial Q \ \Big|\ \vu_b(x) \cdot \vc{n}(x) \geq 0 \right\},
\end{equation*}    
and assume
\begin{equation*}
\Gamma_{\rm out} \ \mbox{is}\ C^2-\mbox{manifold with boundary.}
\end{equation*}
Finally, we suppose
\begin{equation*}
\vr_b \in C(\partial Q),\ \inf_{\partial \Omega}\ \vr_b > 0, 
\end{equation*}
and
\begin{equation*}
\vc{g} \in C(\Ov{Q}; R^3).
\end{equation*}
  
The pressure $p = p(\vr)$ is a continuously differentiable function of the density satisfying
\begin{equation}\label{h1}
\begin{split}
&p \in C^1[0, \infty) \cap C^2(0, \infty),\ p(0) = 0, \\
&\ p'(\vr) \geq \uline{a}\vrho^{\g-1} \ \mbox{ for }\ \vr > 0,\quad
p'(\vr) \leq \Ov{a}\vr^{\gamma - 1} \ \mbox{ for }\ \vr > 1,\ \mbox{ where }\ \gamma > \frac{3}{2}.
\end{split}
\end{equation}	
The hypothesis $\gamma > \frac{3}{2}$ is possibly technical, but so far indispensable both for the existence theory and asymptotic 
compactness of bounded trajectories.

All hypotheses stated in this section will be tacitly assumed throughout the remaining part of the paper. 

\subsection{Weak solutions}
\label{w} 

Extending $\vu_b$ in $Q$, $\vu_b \in C^1(\Ov{Q}; R^3)$, we are ready to introduce the class of finite energy weak solutions 
to the Navier--Stokes system.

\begin{Definition} [Finite energy weak solution] \label{WD1}

We say that $[\vr, \vm]$ is a \emph{finite energy weak solution} of the Navier--Stokes system \eqref{i1}, \eqref{i2}, with the boundary 
conditions \eqref{i3}, \eqref{i4}
in $(\tau, \infty)$, $-\infty \leq \tau < \infty$ if the following holds:

\begin{itemize}

\item {\bf Regularity class.}
\[
\begin{split}
\vr &\in C_{{\rm weak,loc}}((\tau, \infty); L^\gamma (Q)) \cap
L^\gamma_{\rm loc}((\tau, \infty) ; L^\gamma (\partial Q, |\vu_b \cdot \vc{n}| \D S_x )),\ \vr \geq 0,\\
\vm &\in C_{{\rm weak,loc}}((\tau, \infty); L^{\frac{2 \gamma}{\gamma + 1}}(Q; R^3));
\end{split}
\]
there exists a velocity field $\vu$ such that $\vm = \vr \vu$ a.a., and
\[ 
(\vu - \vu_b) \in L^2_{\rm loc}((\tau,\infty); W^{1,2}_0(Q; R^3)).
\]

\item
{\bf Equation of continuity.}

The integral identity
\begin{equation} \label{w2}
\int_\tau^\infty \intO{ \Big[ \vr \partial_t \varphi + \vr \vu \cdot \Grad \varphi \Big] } \dt
=
\int_\tau^\infty \int_{\Gamma_{\rm out}} \varphi \vr \vu_b \cdot \vc{n} \ \D  S_x \ \dt
+
\int_\tau^\infty \int_{\Gamma_{\rm in}} \varphi \vr_b \vu_b \cdot \vc{n} \ \D  S_x \ \dt
\end{equation}
holds for any $\varphi \in C^1_c((\tau,\infty) \times \Ov{\Omega})$. In addition, a renormalized version of \eqref{w2}
\begin{equation} \label{w3}
\begin{split}
\int_\tau^\infty\!\!\! \intO{ \Big[ b(\vr) \partial_t \varphi + b(\vr) \vu \cdot \Grad \varphi -
\Big( b'(\vr) \vr - b(\vr) \Big) \Div \vu \vphi \Big] } \dt = \int_\tau^\infty\!\!\! \int_{\Gamma_{\rm in}} \varphi b(\vr_b) \vu_b \cdot \vc{n} 
\ \D S_x \dt
\end{split}
\end{equation}
holds for any
$\varphi \in C^1_c((\tau,\infty) \times ( {Q} \cup \Gamma_{\rm in} ))$, and any $b \in C^1[0, \infty)$, $b' \in C_c[0, \infty)$.

\item
{\bf Momentum equation.}

The integral identity
\begin{equation} \label{w4}
\int_\tau^\infty \intO{ \Big[ \vr \vu \cdot \partial_t \bfphi + \vr \vu \otimes \vu : \Grad \bfphi
+ p(\vr) \Div \bfphi - \mathbb{S}(\Ds \vu) : \Grad \bfphi + \vr \vc{g} \cdot \bfphi \Big] } \dt = 0
\end{equation}
holds for any $\bfphi \in C^1_c((\tau,\infty) \times Q; R^3)$.

\item {\bf Total energy balance.}

The total energy
satisfies
\begin{equation} \label{w5}
\begin{split}
&- \int_\tau^\infty \partial_t \psi \intO{E \left( \vr, \vu \Big| \vu_b \right) } \dt  +
\int_\tau^\infty \psi \intO{ \mathbb{S}(\Ds \vu) : \Ds \vu } \dt \\
&+ \int_\tau ^\infty \psi \int_{\Gamma_{\rm out}} P(\vr)  \vu_b \cdot \vc{n} \ \D S_x \dt  +
\int_\tau^\infty \psi \int_{\Gamma_{\rm in}} P(\vr_b)  \vu_b \cdot \vc{n} \ \D S_x \dt
\\	
&\leq- \int_\tau^\infty \psi 
\intO{ \left[ \vr \vu \otimes \vu + p(\vr) \mathbb{I} \right]  :  \Ds \vu_b } \dt  + 
\frac{1}{2} \int_\tau^\infty \psi \intO{ {\vr} \vu  \cdot \Grad |\vu_b|^2  }\dt
\\ &+ \int_\tau^\infty \psi \intO{ \mathbb{S}(\Ds \vu) : \Ds \vu_b }\dt  + \int_\tau^\infty \psi \intO{ \vr \vc{g} \cdot (\vu -
\vu_b) } \dt
\end{split}
\end{equation}
for any
$\psi \in C^1_c(\tau, \infty)$, $\psi \geq 0$; in addition,
\begin{equation} \label{w6}
\limsup_{t \to \tau +} \intO{E \left( \vr, \vu \Big| \vu_b \right) }< \infty.
\end{equation}

\end{itemize}

\end{Definition}

The solutions introduced in Definition \ref{WD1} are defined on an \emph{open} time interval $(\tau, \infty)$, including 
the case $\tau = - \infty$. In particular, the initial state of the system at the time $\tau$, if finite, has not been specified. 
On the other hand, as the energy at time $\tau$ is bounded, cf. \eqref{w6}, it is easy to show that $[\vr, \vm]$ can be extended as 
\[
\vr \in C_{{\rm weak,loc}}([\tau, \infty); L^\gamma (Q)),\ 
\vm \in C_{{\rm weak,loc}}([\tau, \infty); L^{\frac{2 \gamma}{\gamma + 1}}(Q; R^3));
\] 
whence the initial state is well defined. The total energy, being a convex l.s.c. function of $[\vr, \vm]$,
is weakly l.s.c. on the associated phase space, and, in general, 
\begin{equation*}
\intO{ E \left( \vr(\tau, \cdot), \vm(\tau, \cdot) \Big| \vu_b \right) } \leq 
\liminf_{t \to \tau+}\intO{ E \left( \vr(t, \cdot), \vm(t, \cdot) \Big| \vu_b \right) }. 
\end{equation*}
Hypothetically, the energy may experience an ``initial jump'', meaning the solution 
\[
t \mapsto [\vr(t, \cdot), \vm(t, \cdot)] 
\]
may not be (right) continuous at $t = \tau$ with respect to the \emph{strong topology of} 
$L^\gamma(Q) \times L^{\frac{2 \gamma}{\gamma + 1}}(Q; R^3)$. As we shall see below, such a scenario is excluded for the stationary statistical solutions a.s., see Section~\ref{ss:cont-en}.

The \emph{existence} of finite energy weak solutions for $\tau$ finite and given initial data has been established by Chang, Jin, and 
Novotn\' y \cite{ChJiNo}, see also \cite{FanFei} for the necessary modification of the proof to accommodate the ``differential form'' 
of the energy inequality \eqref{w5}.

\subsubsection{Strong continuity of the density}
The following result is an extension of \cite[Proposition 4.3]{EF70} to the case of general boundary conditions.
 
\begin{Proposition}[Strong time continuity of the density] \label{wP1}

Let $[\vr, \vm]$ be a finite energy weak solution defined in $(\tau , \infty)$ and $[t_1,t_2] \subset (\tau , \infty)$ be a time 
interval. 

Then $\vr \in C([t_1,t_2]; L^1(Q))$.

\end{Proposition} 

\begin{proof}
To begin with, we take a smooth function $T$ such that
\begin{equation} \label{def:trunc-op}
\begin{split}
&T\in C^\infty[0,\infty)\,,\qquad T \ \mbox{ is increasing and concave on }\ [0,\infty)\,, \\
&\qquad T(z)\,=\,z\quad\mbox{ for }\ z\in[0,1]\qquad\quad\mbox{ and }\qquad\quad T(z)\,=\,2\quad \mbox{ for }\ z\geq3\,.
\end{split}
\end{equation}
For any integer $k\in\N\setminus\{0\}$, we introduce the cut-off functions $T_k$ by the formula
\begin{equation} \label{def:T_k}
T_k(z)\,=\,k\,T\left(\frac{z}{k}\right)\,.
\end{equation}
Then, for all $t\in(\t,\infty)$ and almost every $x\in Q$, we define $b_k$ as
$$
b_k(t,x)\,=\,T_k\big(\vr(t,x)\big)\,.
$$
Notice that the hypotheses formulated on $T$ allow to use the renormalized continuity equation \eqref{w3} on $T_k$, for any $k$ fixed. So, we infer that $b_k=T_k(\vr)$
solves \eqref{w3}.

Let now $[t_1,t_2]\subset(\t,\infty)$ be a compact time interval. As a consequence of our definitions, of the properties of $\vr$ and of \eqref{w3} for $b_k=T_k(\vr)$,
for any $k\in\N\setminus\{0\}$ we get
\begin{equation} \label{eq:propr-b}
b_k \in L^\infty((t_1,t_2) \times Q) \qquad \mbox{ and }\qquad
b_k \in C_{\rm weak}([t_1,t_2]; L^q(Q)) \quad \mbox{ for any }\ 1 \leq  q < \infty\,.
\end{equation}
Next, we observe that the following convergence property holds true, in the limit $k\ra\infty$:
\begin{equation} \label{conv:trunc}
b_k\,=\,T_k(\vr)\,\longrightarrow\,\vr\qquad\qquad\mbox{ strongly in }\quad L^\infty([t_1,t_2];L^1(Q))\,. 
\end{equation}
In order to see this, we start by remarking that, for any $t\in[t_1,t_2]$ and any $k\in \N\setminus\{0\}$, by virtue of the Chebyshev inequality we have
$$
\mc L\left\{x\in Q\;\Big|\quad \vrho(t,x)\geq k\, \right\}\,\leq\,\frac{1}{k^\g}\,\int_{\left\{x\in Q|\,\vrho(t,x)\geq k \right\}}{\Big(\vr(t,x)\Big)^\g}\,\dx\,,
$$
where we have denoted by $\mc L(A)$ the Lebesgue measure of a set $A\subset Q$.
Then, for any $t\in[t_1,t_2]$, after setting
$$
A_k\,\equiv\,\left\{x\in Q\;\Big|\quad \vrho(t,x)\geq k\, \right\}\,,
$$
we can estimate
\begin{align*}
\intO{\left|b_k(t,\cdot)\,-\,\vr(t,\cdot)\right|}\,=\,\int_{A_k}\left|b_k(t,\cdot)\,-\,\vr(t,\cdot)\right|\,\dx\,&\leq\,\int_{A_k}b_k(t,\cdot)\,\dx\,+\,\int_{A_k}\vr(t,\cdot)\,\dx \\
&\leq\,2\,k\,\mc L(A_k)\,+\,\|\vr(t,\cdot)\|_{L^\g}\,\Big(\mc L(A_k)\Big)^{1/\g'} \\
&\leq\, k^{-(\g-1)}\,\left(2\,+\,\left\|\vr(t,\cdot)\right\|_{L^\g}^{\g-1}\right)\,,
\end{align*}
where, in the last step, we have used Chebyshev inequality and the fact that $\g/\g'\,=\,\g-1$. Taking the $\sup$ over $[t_1,t_2]$ of both sides of the previous inequality
completes the proof of \eqref{conv:trunc}.

Owing to the convergence property \eqref{conv:trunc}, the proof of the proposition boils down to showing that $b_k \in C([t_1,t_2]; L^1(Q))$ for any $k\in\N\setminus\{0\}$.
Note that the instantaneous values $b_k(t, \cdot) \in L^1(Q)$ are well defined for any $t$ by its weakly continuous representative. 
Since the next argument does not depend on the value of $k\geq 1$, for the rest of the proof we drop the index $k$ from the notation and simply write $b=b(\vr)$ for $b_k=b_k(\vr)=T_k(\vr)$.

So, let $b$ enjoy properties \eqref{eq:propr-b}, and consider its regularization $b_\ep = \theta_\ep * b$ by spatial convolution.
Standard properties of mollifying kernels and the same argument which led to \eqref{eq:propr-b} yield
\begin{equation} \label{w8} 
\begin{split}
&b_\ep \in C([t_1,t_2]; L^q(Q)) \ \mbox{and}\\
& \ b_\ep (t, \cdot) \to b(t, \cdot) \in L^q_{\rm loc}(Q) \ \mbox{as}\ 
\ep \to 0, \quad \mbox{for any}\ t \in [t_1,t_2],\  1 \leq q < \infty.
\end{split}
\end{equation}

It follows from the renormalized equation of continuity that 
\begin{equation} \label{w9}
\partial_t b_\ep + \Div ( b_\ep \vu) + \theta_\ep * f = r_\ep \ \mbox{in}\ (a,b) \times O,\ O \subset \Ov{O} \subset Q, 
\end{equation}
for all $\ep > 0$ small enough, where 
\[
f = \big(b'(\vr) \vr - b(\vr)\big) \Div \vu \in L^2((t_1,t_2) \times Q),\ 
r_\ep \to 0 \ \mbox{in} \ L^r((t_1,t_2) \times Q) \ \mbox{whenever}\ 1 \leq r < 2, 
\]
cf. DiPerna and Lions \cite{DL}. Using the fact that $b$ is bounded, we deduce from \eqref{w9}  
\begin{equation} \label{w10}
\partial_t F(b_\ep)+ \Div ( F(b_\ep) \vu) + F'(b_\ep) \theta_\ep * f = F'(b_\ep) r_\ep + 
\left( F(b_\ep) - F'(b_\ep) b_\ep \right) \Div \vu 
\end{equation}  
for any continuously differentiable convex function $F$. On the one hand, in view of \eqref{w8}, 
\[
F(b_\ep (t, \cdot)) \to F(b(t, \cdot)) \ \mbox{in}\ L^q_{\rm loc}(Q) \ \mbox{for any}\ t \in [t_1,t_2].
\]
On the other hand, in view of \eqref{w10}, we get (as done for \eqref{eq:propr-b} above) that
\[
F(b_\ep) \to \Ov{F} \ \mbox{in} \ C_{\rm weak}([t_1,t_2]; L^q_{\rm loc}(Q)).
\]
Necessarily, $\Ov{F} = F(b) \in C_{\rm weak}([t_1,t_2]; L^q_{\rm loc}(Q))$ for any differentiable convex function.
In particular, we deduce that $b^2\in C_{\rm weak}([t_1,t_2]; L^q_{\rm loc}(Q))$; combining this property with \eqref{eq:propr-b}, we infer that
$$
b \in C([t_1,t_2]; L^2_{\rm loc}(Q))\,,\qquad\mbox{ whence }\quad b \in C([t_1,t_2]; L^1_{\rm loc}(Q))\,.
$$
The time continuity with values in $L^1(Q)$ then follows by using the fact that $b\in L^\infty((t_1,t_2)\times Q)$.

As already remarked, combining this fact together with \eqref{conv:trunc}, we finally deduce the time continuity of $\vr$ with values in $L^1(Q)$.
\end{proof}

\subsection{Phase space, trajectory space, $\omega$--limit sets} \label{ss:Traj}

We introduce the \emph{phase space}
\[
H = L^1(Q) \times W^{-k,2}(Q; R^3),
\] 
the \emph{trajectory space}
\[
\mathcal{T} \equiv \left\{ (r, \vc{w}) \ \Big| \ r \in C(R; L^{1}(Q)),\ 
\vc{w} \in C(R; W^{-k,2}(Q; R^3)) \right\}
\]
and the metric on the trajectory space, 
\[
d_\mathcal{T} \Big[ (r_1, \vc{w}_1); (r_2, \vc{w}_2) \Big] \equiv 
\sum_{M=1}^\infty \frac{1}{2^M}
 \frac{ \sup_{t \in [-M, M]}\left\| (r_1, \vc{w}_1)(t, \cdot) - (r_2, \vc{w}_2)(t, \cdot) \right\|_{L^{1}\times W^{-k,2}} }
{1 + \sup_{t \in [-M, M]}\left\| (r_1, \vc{w}_1)(t, \cdot) - (r_2, \vc{w}_2)(t, \cdot) \right\|_{L^{1}\times W^{-k,2}} }.
\]

Given regularity of the spatial domain $Q$, we have the compact embedding $W^{k,2}(Q) \hookrightarrow\hookrightarrow C(\Ov{Q})$ for $k > \frac{3}{2}$ -- 
a condition assumed hereafter.
It is standard to observe  that 
\begin{itemize} 
\item
$[\mathcal{T}; d_{\mathcal{T}}]$ is a Polish space; 
 \item 
if $[\vr, \vm]$ is a finite energy weak solution in the sense of Definition \ref{WD1}, then the trajectory $[\tvr, \tvm]$,
defined as 
\[
[\tvr(t, \cdot), \tvm(t, \cdot)] = \left\{ \begin{array}{l} \left[\vr(t, \cdot), \vm(t, \cdot) \right] \ \mbox{if}\ t > \tau, \\ 
\left[ \vr(\tau, \cdot), \vm(\tau, \cdot) \right] \ \mbox{if}\ t \leq \tau, \end{array} \right. 
\]
belongs to $\mathcal{T}$ as soon as $k > \frac{3}{2}$.
\end{itemize}

Next, we recall the definition of the space of bounded entire solutions,
\[
\mathcal{U}[\Ov{E}] = \left\{ [\vr, \vm] \ \Big|\ 
[\vr, \vm] \ \mbox{finite energy weak solution in} \ R \ \mbox{such that} \ \sup_{t \in R} \intO{ E\left( \vr, \vm \ \Big| \vu_b \right) }
\leq \Ov{E}
\right\}.
\]
Notice that, owing to Proposition \ref{wP1} and the compact embedding $W^{k,2}(Q) \hookrightarrow\hookrightarrow C(\Ov{Q})$ for $k > \frac{3}{2}$, any finite energy weak solution
$[\vr,\vm]$ in $R$ is continuous in time with respect to the strong topology of $H$. Thus, we have $\mc U[\Ov{E}]\subset\mc T$.

Finally, let $[\vr, \vm]$ be a solution of the Navier--Stokes system in $(0, \infty)$ such that
\begin{equation} \label{T1}
\sup_{t > 0} \intO{ E \left( \vr, \vm \Big| \ \vu_b \right)(t, \cdot) } \leq \Ov{E}.
\end{equation}
Extending 
\[
\vr(t, \cdot) = \vr(0, \cdot), \ \vm(t, \cdot) = \vm(0, \cdot) 
\]
we may assume $[\vr, \vm] \in \mathcal{T}$. Note that the $\omega-$limit set
\[
\omega[\vr, \vm] = \left\{ [r, \vc{w}] \in \mathcal{T}\ \Big| \ \mbox{there exists}\ T_n \to \infty 
\ \mbox{such that} \ [\vr, \vm](\cdot + T_n) \to (r, \vc{w}) \ \mbox{in}\ \mathcal{T} \right\}
\]
is a subset of the trajectory space $\mathcal{T}$ in contrast with a more conventional definition that 
identifies the $\omega-$limit set with the set of all possible limits of $[\vr(T_n, \cdot), \vm(T_n, \cdot)]$ in the \emph{phase space} 
$L^{1}(Q) \times W^{-k,2}(Q; R^3)$.

\section{Global solutions with bounded energy} \label{s:energy}

The main effect of general inflow/outflow boundary conditions \eqref{i3} is that there may be energy or mass exchanges between the system and the exterior. In particular, the energy of solutions
may not be globally bounded in time.

In the first part of this section, we prove a fundamental asymptotic compactness result for trajectories having globally bounded energy.
In Subsection \ref{ss:existence}, instead, we show that, under suitable assumptions on the boundary data, such trajectories indeed exist.

\subsection{Asymptotic compactness of bounded trajectories} \label{a}

The following result is absolutely crucial for the existence of stationary statistical solutions generated by bounded trajectories. 
It is new in the context of general inflow/outflow boundary conditions and as such may be of independent interest. We point out that the
proof cannot be done as a simple adaptation of the existence theory, as the crucial information on compactness of initial densities 
is missing. Instead, the uniform temporal decay of the density oscillation defect measure, first observed in \cite{EF53}, must be 
shown.

\begin{Theorem}[Asymptotic compactness] \label{aP1}

Let $\left\{ \vr_n, \vm_n \right\}_{n=1}^\infty$ be a sequence of finite energy weak solutions of the Navier--Stokes system in the sense of Definition \ref{WD1} defined in $(\tau_n, \infty) \times Q$,
with
\[
\tau_n \geq - \infty, \ \tau_n \to - \infty \ \mbox{as} \ n \to \infty,
\]
satisfying 
\[
\intO{ E \left( \vr_n, \vm_n \Big| \vu_b \right)(t, \cdot) } \leq \Ov{E} < \infty \ \mbox{for all}\ t > \tau_n,\ 
n = 1,2,\dots 
\]

Then there is a subsequence (not relabeled) such that the extended trajectories  
\begin{equation} \label{eq:extension}
\begin{split}
\vr_n(t,x) &= \vr_n(t, x), \ \vm_n(t,x) = \vm_n(t, x) \ \mbox{for}\ t > \tau_n ,\\ 
\vr_n(t,x) &= \vr_n(\tau_n +, x), \ \vm_n (t,x) = \vm_n (\tau_n +, x) \ \mbox{for}\ t \leq \tau_n, \ x \in Q,
\end{split}
\end{equation}
admit a limit in $[\mathcal{T}, d_{\mathcal{T}}]$, i.e. there exists $[\vr,\vm]\in [\mc T , d_{\mc T}]$ such that
\[
[\vr_n, \vm_n] \to [\vr, \vm] \ \mbox{ in } \ [\mathcal{T}, d_{\mathcal{T}}]. 
\]
In addition, $[\vr, \vm]$ is a finite energy weak solution of the Navier--Stokes system in $R \times Q$, in the sense of Definition \ref{WD1},
with $[\vr,\vm]\in\mc U[\oline{E}]$.
\end{Theorem}

\subsubsection{Proof of Theorem \ref{aP1}}
This paragraph is devoted to the proof of the previous statement. We proceed in several steps.

\medskip

\noindent{\bf Step 1 (uniform bounds and convergence):}

We assume that all the trajectories $[\vr_n, \vm_n]$ have been extended on $R\times Q$ by formula \eqref{eq:extension}.
As the energy is uniformly bounded, we get immediately 
\begin{equation} \label{a1}
\| \vr_n (t, \cdot) \|_{L^\gamma(Q)} + \intO{ \frac{|\vm_n|^2}{\vr_n}(t, \cdot) } \aleq 1 
\end{equation}
uniformly for $t \in R$, and, by means of H\" older inequality, 
\begin{equation} \label{a2}
\| \vm_n (t, \cdot) \|_{L^{\frac{2 \gamma}{\gamma + 1}}(Q; R^3)} \aleq 1 
\end{equation} 
uniformly for $t \in R$. As $k > \frac{3}{2}$ we conclude there there is a compact set $\mathcal{K} \subseteq W^{-k,2}(Q) 
\times W^{-k,2}(Q; R^3)$ such that
\begin{equation} \label{a3}
[\vr_n(t, \cdot), \vm_n(t, \cdot)] \in \mathcal{K} \ \mbox{for all}\ t \in R,\ n=1,2,\dots
\end{equation}

Next observe that the energy inequality \eqref{w5} yields the uniform bounds
\begin{equation} \label{a4}
\int_{-M}^M\!\! \intO{ \mathbb{S}(\Ds \vu_n) : \Ds \vu_n } \dt + 
\int_{-M}^M\!\! \int_{\partial Q} P(\vr_n)  \ |\vu_b \cdot \vc{n}| \ \D S_x \dt \leq C(M) \ \mbox{for any}\ 0 < M < - \tau_n,
\end{equation}
where we have set $\vr_n = \vr_b$ on $\Gamma_{\rm in}$ and  $n$ is sufficiently large so that $\tau_{n}<0$.
Consequently, by means of Korn--Poincar\' e inequality, 
\begin{equation*}
\int_{-M}^M \| \vu_n \|^2_{W^{1,2}(Q; R^d)} \dt \leq c(M) \ \mbox{for any}\ 0 < M < - \tau_n.
\end{equation*}
Finally, it follows from the field equations \eqref{w2}, \eqref{w4} that 
\begin{equation} \label{a6}
\left\| \frac{\D}{\dt}\!\! \intO{ \vr_n \varphi } \right\|_{L^\infty(-M,M)} + 
\left\| \frac{\D}{\dt}\!\! \intO{ \vm_n \cdot \bfphi } \right\|_{L^2(-M,M)} \leq c(M, \varphi, \bfphi) \ 
\mbox{for any}\ 0 < M < - \tau_n
\end{equation}
and any test functions $\varphi \in C^1_c(Q)$, $\bfphi \in C^1_c(Q; R^3)$.

In view of \eqref{a3}, \eqref{a6}, the abstract Arzel\` a--Ascoli Theorem, and the uniform bounds \eqref{a1}, \eqref{a2},
we may infer that 
\begin{equation} \label{a7}
\begin{split}
&\vr_n \to \vr \ \mbox{in}\ C_{{\rm weak}}([-M; M]; L^\gamma (Q)),\\
&\ \vm_n \to \vm \ \mbox{in}\ C_{{\rm weak}}([-M; M]; L^{\frac{2 \gamma}{\gamma + 1}} (Q; R^3)),\ 
M > 0 \ \mbox{arbitrary},
\end{split}
\end{equation}
passing to a suitable subsequence as the case may be. In particular,   
\[
[\vr_n, \vm_n] \to [\vr, \vm] \ \mbox{in} \ [\mathcal{T}, d_{\mathcal{T}}].
\]

\medskip

\noindent{\bf Step 2 (velocity and the limit in the convective terms):}

It remains to show that $[\vr, \vm]$ is a weak solution of the Navier--Stokes system in $R \times Q$. First observe that 
$\vr \geq 0$, being a (weak) limit of non--negative functions. Moreover, as the total energy $E(\vr, \vm \Big| \vu_b )$ is a convex l.s.c.
function,  we deduce 
\[
\intO{ E \left( \vr, \vm \ \Big| \ \vu_b \right)(t, \cdot) } \leq \Ov{E} \ \mbox{for}\ t \in R.
\]

In view of \eqref{a7}, we may suppose, extracting a suitable subsequence, 
\begin{equation*}
\vu_n \to \vu \ \mbox{weakly in}\ L^2(-M,M; W^{1,2}(Q; R^3)),\ \mbox{for any}\ M > 0.
\end{equation*}
Moreover, as $\gamma > \frac{3}{2}$, we deduce from \eqref{a7} that 
\begin{equation} \label{a9}
\vm = \vr \vu,\ 
\vr_n \vu_n \otimes \vu_n \to \vr \vu \otimes \vu \ 
\mbox{weakly in}\ L^\alpha ((-M; M) \times Q; R^{3 \times 3}) \ \mbox{for some}\ \alpha > 1.
\end{equation}
These and several other arguments used in what follows are nowadays quite well understood, and we refer to 
\cite[Chapter 4]{EF70} for details.

\medskip

\noindent{\bf Step 3 (equation of continuity):} 

With the previous estimates at hand, it is a routine matter to perform the limit in the equation of continuity 
\eqref{w2}. Indeed the bound \eqref{a4} yields, up to a subsequence, 
\[
\vr_n \to \vr \ \mbox{weakly in}\ L^\gamma ((-M; M) \times \partial Q; |\vu_b \cdot \vc{n}| \D S_x );
\]
whence 
\begin{equation*}
\int_R \intO{ \Big[ \vr \partial_t \varphi + \vr \vu \cdot \Grad \varphi \Big] } \dt
=
\int_R \int_{\Gamma_{\rm out}} \varphi \vr \vu_b \cdot \vc{n} \ \D  S_x \ \dt
+
\int_R \int_{\Gamma_{\rm in}} \varphi \vr_b \vu_b \cdot \vc{n} \ \D  S_x \ \dt
\end{equation*}
for any $\varphi \in C^1_c(R \times \Ov{Q})$.

\medskip

\noindent{\bf Step 4 (equi--integrability of the pressure):} 

In order to perform the limit in the momentum and the energy balance, we need to establish equi--integrability of the pressure and the pressure potential with respect to the space variable.
This can be done exactly as in \cite[Section 4]{FanFei} (see also \cite[Section 6]{FP9}) obtaining that the sequences $\{ p(\vr_n) \}_{n=1}^\infty$ as well as $\{ P(\vr_n) \}_{n=1}^\infty$
are equi--integrable in  $(-M, M) \times Q$ for any $M > 0$. In particular, we may assume 
\begin{equation} \label{a11}
p(\vr_n) \to \Ov{p(\vr)},\ 
P(\vr_n) \to \Ov{P(\vr)} \ \mbox{weakly in}\ L^1((-M,M) \times Q) \ \mbox{for any}\ M > 0.
\end{equation}
Here and everywhere in the rest of this section, we use the symbol $\Ov{G(\vr)}$ to denote a weak $L^1-$limit of a sequence $\{ G(\vr_n) \}_{n=1}^\infty$.

Similarly, in view of \eqref{a9}, we have
\begin{equation} \label{a12}
E \left(\vr_n, \vm_n \Big| \vu_b \right) \equiv 
E \left(\vr_n, \vu_n \Big| \vu_b \right) \to \left( \frac{1}{2} \vr |\vu|^2 + \Ov{P(\vr)} \right) 
\ \mbox{weakly in}\ L^1((-M, M) \times Q).
\end{equation}

\medskip

\noindent{\bf Step 5 (weak convergence):}

With \eqref{a11}, \eqref{a12} at hand, we may perform the limit in the remaining family of the field equations: 
\begin{equation} \label{a13}
\begin{split}
\int_R \intO{ \Big[ \Ov{b(\vr)} \partial_t \varphi + \Ov{b(\vr)} \vu \cdot \Grad \varphi -
\Ov{ \Big( b'(\vr) \vr - b(\vr) \Big) \Div \vu } \Big] } \dt = \int_R  \int_{\Gamma_{\rm in}} \varphi b(\vr_b) \vu_b \cdot \vc{n} 
\ \D S_x \dt
\end{split}
\end{equation}
for any
$\varphi \in C^1_c(R \times ( Q \cup \Gamma_{\rm in} ))$, and any $b \in C^1[0, \infty)$, $b' \in C_c[0, \infty)$;
\begin{equation*}
\int_R \intO{ \Big[ \vr \vu \cdot \partial_t \bfphi + \vr \vu \otimes \vu : \Grad \bfphi
+ \Ov{p(\vr)} \Div \bfphi - \mathbb{S}(\Ds \vu) : \Grad \bfphi + \vr \vc{g} \cdot \bfphi \Big] } \dt = 0
\end{equation*}
for any $\bfphi \in C^1_c(R \times {Q}; R^3)$; and
\begin{equation} \label{a15}
\begin{split}
&- \int_R \partial_t \psi \intO{ \Ov{ E \left( \vr, \vm \Big| \vu_b \right) } } \dt  +
\int_R \psi \intO{ \mathbb{S}(\Ds \vu) : \Ds \vu } \dt \\
&+ \int_R  \psi \int_{\Gamma_{\rm out}} P(\vr)  \vu_b \cdot \vc{n} \ \D S_x \dt  +
\int_R \psi \int_{\Gamma_{\rm in}} P(\vr_b)  \vu_b \cdot \vc{n} \ \D S_x \dt
\\	
&\leq- \int_R \psi 
\intO{ \left[ \vr \vu \otimes \vu + \Ov{p(\vr)} \mathbb{I} \right]  :  \Ds \vu_b } \dt  + 
\frac{1}{2} \int_R \psi \intO{ {\vr} \vu  \cdot \Grad |\vu_b|^2  }\dt
\\ &+ \int_R \psi \intO{ \mathbb{S}(\Ds \vu) : \Ds \vu_b }\dt  + \int_R \psi \intO{ \vr \vc{g} \cdot (\vu -
\vu_b) } \dt
\end{split}
\end{equation}
for any
$\psi \in C^1_c(\tau, \infty)$, $\psi \geq 0$. Here, we have systematically used the symbol $\Ov{G(\vr, \vm)}$ to denote a weak $L^1-$limit of a sequence $\{ G(\vr_n, \vm_n) \}_{n=1}^\infty$, or, equivalently,
the expected value of $G$ with respect to a Young measure associated to $\{ \vr_n, \vm_n \}_{n=1}^\infty$. In general, such a process requires extracting a subsequence if necessary.

\medskip

\noindent{\bf Step 6 (a.e. convergence of density):}

The next step of the proof is showing strong (a.a. pointwise) convergence of the sequence of densities $\{ \vr_n \}_{n=1}^\infty$, or, equivalently, removing the upper bars over the nonlinearities in \eqref{a13}--\eqref{a15}. 
The key issue is to show that
\begin{equation} \label{eq:r_log-r}
\intO{\Big( \oline{\vr\,\log(\vr)}\,-\,\vr\,\log(\vr) \Big)(t,x)}\equiv0\qquad\mbox{ for all }\ t\in R\,.
\end{equation}
Indeed, thanks to the strict convexity of the function $\vr\mapsto\vr\log(\vr)$ on $[0,\infty)$, the previous property would imply (see \cite[Theorem 2.11]{EF70}), up to a subsequence, the required a.e. convergence.
In order to obtain \eqref{eq:r_log-r}, one would like to use the function $b(\vr)=\vr\log(\vr)$ in the renormalized continuity equation \eqref{w3} for the limit density $\vr$, and compare this with
\eqref{a13}. Unluckily, this strategy cannot work directly, due to the lack of integrability of $b'(\vr)\vr-b(\vr)=\vr$, which does not allow to properly define the product
$\big(b'(\vr)\vr-b(\vr)\big)\Div\vu$. Instead, one has to implement an approximation procedure.

To begin with, we recall the so--called Lions identity (cf. Lions \cite{LI4}),
\begin{equation} \label{a16}
\Ov{p(\vr) b(\vr)} - \Ov{p(\vr)} \ \Ov{b(\vr)} = \left( \lambda + \frac{2 d - 2}{d}\mu \right) \left( 
\Ov{b(\vr) \Div \vu} - \Ov{b(\vr)} \Div \vu \right)
\end{equation}
for any $b \in C^1[0, \infty)$, $b' \in C_c[0, \infty)$, where $d$ denotes the space dimension (recall that $d = 3$ for us).
As \eqref{a16} is of local character, the proof ``does not see'' the boundary conditions and can be done exactly as in \cite[Chapter 6]{EF70}. 

It follows form \eqref{a16} that the \emph{oscillation defect measure} introduced in \cite[Chapter 6]{EF70} is bounded: specifically, after introducing the cut-off functions $T_k$ as in
\eqref{def:trunc-op}--\eqref{def:T_k}, we have
\begin{equation} \label{a16a}
\sup_{k \geq 1} \left( \limsup_{n \to \infty} \int_{-M}^M \intO{ \left| T_k(\vr_n) - T_k(\vr) \right|^{\gamma + 1} } \dt \right)
\leq c(M)
\end{equation}
for any $M > 0$. We refer to \cite[Chapter 6]{EF70} for details.

Relations \eqref{a16}, \eqref{a16a} can be used to show that the limit functions $\vr$, $\vu$ satisfy the renormalized equation of continuity 
\eqref{w3}. The original proof from \cite[Chapter 6]{EF70} has been adapted in a nontrivial way to the inflow--outflow boundary conditions by Chang, Jin, and Novotn\' y \cite[Section 3.2]{ChJiNo}.
We may therefore conclude that   
\begin{equation} \label{a17}
\begin{split}
\int_R\intO{ \Big[ {b(\vr)} \partial_t \varphi + {b(\vr)} \vu \cdot \Grad \varphi -
\varphi { \Big( b'(\vr) \vr - b(\vr) \Big) \Div \vu } \Big] } \dt = \int_R \int_{\Gamma_{\rm in}} \varphi b(\vr_b) \vu_b \cdot \vc{n} 
\ \D S_x \dt,
\end{split}
\end{equation}
for any
$\varphi \in C^1_c(R \times ( Q \cup \Gamma_{\rm in} ))$, and any $b \in C^1[0, \infty)$, $b' \in C_c[0, \infty)$.

Now observe that validity of both \eqref{a13} and \eqref{a17} can be extended to the function 
\[
b(\vr) = L_k(\vr) , \ L_k'(\vr) \vr - L_k(\vr) = T_k(\vr),\ T_k(\vr) = \min \{ \vr, k \},
\]
that is a compactly supported perturbation of an affine function.
Consequently, subtracting \eqref{a17} from \eqref{a13} gives rise to  
\begin{equation} \label{a18}
\begin{split}
\int_R \intO{ \left(  \left[ \Ov{L_k(\vr)} - L_k(\vr) \right]\partial_t \varphi + \left[ \Ov{L_k(\vr)} - L_k(\vr) \right]  \vu \cdot \Grad \varphi + \varphi \left[ T_k(\vr) \Div \vu - \Ov{T_k(\vr) \Div \vu } \right]
 \right) } \dt = 0
\end{split}
\end{equation}
for any
$\varphi \in C^1_c(R \times ( Q \cup \Gamma_{\rm in} ))$.

Our next goal is to extend validity of \eqref{a18} to spatially homogeneous test functions, meaning $\varphi = 
\psi (t)$, $\psi \in C^1_c(R)$. To this end, consider a sequence of test functions $\{ \phi_n \}_{n=1}^\infty$ such that 
\[
\phi_n \in C^1_c(Q \cup \Gamma_{\rm in} ),\ 0 \leq \phi_n \leq 1,\ 
\phi_n = 1 \ \mbox{if} \ {\rm dist}[x, \partial Q] \geq \frac{1}{n},\ 
|\Grad \phi_n| \aleq n .
\]
Plugging $\varphi(t,x) = \psi(t) \phi_n (x)$ in \eqref{a18} we easily observe that 
\[
\begin{split}
\int_R \intO{ \left[ \Ov{L_k(\vr)} - L_k(\vr) \right]\partial_t \psi \phi_n  } \dt &\to 
\int_R \intO{ \left[ \Ov{L_k(\vr)} - L_k(\vr) \right] \partial_t \psi } \dt, \\ 
\int_R \intO{ \psi \phi_n \left[ T_k(\vr) \Div \vu - \Ov{T_k(\vr) \Div \vu} \right]  } \dt &\to 
\int_R \intO{\psi \left[ T_k(\vr) \Div \vu - \Ov{T_k(\vr) \Div \vu} \right]  } \dt
\end{split}
\]
by means of Lebesgue dominated convergence theorem. 
Finally, we write 
\[
\begin{split}
\int_R \intO{ \psi \left[ \Ov{L_k(\vr)} - L_k(\vr) \right] \vu \cdot \Grad \phi_n } \dt 
&= \int_R \intO{ \psi \left[ \Ov{L_k(\vr)} - L_k(\vr) \right] (\vu - \vu_b) \cdot \Grad \phi_n } \dt\\
&+ \int_R \intO{ \psi \left[ \Ov{L_k(\vr)} - L_k(\vr) \right] \vu_b \cdot \Grad \phi_n } \dt.
\end{split}
\]
As $(\vu - \vu_b) \in L^2_{\rm loc}(R; W^{1,2}_0(Q; R^3))$, we have 
\[
\frac{ (\vu - \vu_b) }{{\rm dist}[\cdot, \partial Q] } \in L^2_{\rm loc}(R; L^2(Q)). 
\] 
Moreover, since $L_k$ is a bounded perturbation of an affine function, the defect $\Ov{L_k(\vr)} - L_k(\vr)$
is bounded for any fixed $k$, and we obtain 
\[
\int_R \intO{ \psi \left[ \Ov{L_k(\vr)} - L_k(\vr) \right] (\vu - \vu_b) \cdot \Grad \phi_n } \dt \to 0 
\ \mbox{as}\ n \to \infty.
\] 

As for the last integral, given the assumed regularity of $\partial Q$ and its component $\Gamma_{\rm out}$, we introduce the closest point mapping, 
\[
\pig (x) = x_\Gamma \in {\Gamma_{\rm out}},\ 
| x - x_{\Gamma} | = \inf_{\widehat x \in \Gamma_{\rm out} }|x - \widehat{x}|,
\]
and the distance function
\[
x \in \Omega \mapsto {\rm dist}[x, \Gamma_{\rm in}] \in W^{1, \infty}(Q), 
\]
with its gradient
\[
\Grad {\rm dist}[x, \Gamma_{\rm out}] = \frac{ x - \pig(x) }{|x - \pig(x)|} \ \mbox{for a.a.}\ x \in Q.
\]
Finally, we observe that if  
\[
\pig(x) \in {\rm int}[\Gamma_{\rm out}], \ \mbox{meaning}\ \vu_b(\pig(x)) \cdot 
\vc{n}(\pig(x)) > 0, 
\]
then 
\begin{equation} \label{h4}
\frac{ x - \pig(x) }{|x - \pig(x)|} = \Grad {\rm dist}[x, \Gamma_{\rm out}] = 
- \vc{n}(\pig(x)).
\end{equation}

Now, consider  
\[
\phi_n(x) = \min \left\{ n {\rm dist}[x , \Gamma_{\rm in}]; 1 \right\}  \in  W^{1, \infty}(Q).
\]
cf. \cite[Section 7.2]{FeiNov2017} or Chang, Jin, and Novotn\' y \cite{ChJiNo}. 
Accordingly, after setting
$$
U\left(\Gamma_{\rm in}, \frac{1}{n}\right)\,\equiv\,\left\{x\in Q\;\Big|\quad {\rm dist}[x,\Gamma_{\rm in}]\leq\frac{1}{n}\right\}\,,
$$
we have
\[
\begin{split}
&\intO{ \left[ \Ov{L_k(\vr)} - L_k(\vr) \right] \vu_b \cdot \Grad \phi_n } = 
n \int_{U(\Gamma_{\rm in}, \frac{1}{n}) } \left[ \Ov{L_k(\vr)} - L_k(\vr) \right] \vu_b(x) \cdot \frac{ x- \pig(x)}{|x - \pig(x)|}
\ \dx\\ 
&\qquad= n \int_{U(\Gamma_{\rm in}, \frac{1}{n}) } \left[ \Ov{L_k(\vr)} - L_k(\vr) \right] (\vu_b(x) - 
\vu_b (\pig(x))) \cdot \frac{ x- \pig(x)}{|x - \pig(x)|} \dx \\ 
&\qquad\qquad+ n \int_{U(\Gamma_{\rm in}, \frac{1}{n}) } \left[ \Ov{L_k(\vr)} - L_k(\vr) \right] \vu_b(\pig(x)) \cdot \frac{ x- \pig(x)}{|x - \pig(x)|}
\ \dx.
\end{split}
\]
As $\vu_b$ is continuously differentiable in $\Ov{Q}$, we have 
\[
|\vu_b(x) - 
\vu_b (\pig(x))| \aleq \frac{1}{n} \ \mbox{for}\ x \in U\left(\Gamma_{\rm in}; \frac{1}{n}\right),
\]
whence
\[
n \int_{U(\Gamma_{\rm in}, \frac{1}{n}) } \left[ \Ov{L_k(\vr)} - L_k(\vr) \right] (\vu_b(x) - 
\vu_b (\pig(x))) \cdot \frac{ x- \pig(x)}{|x - \pig(x)|} \dx \to 0 \ \mbox{as}\ n \to \infty.
\]
Finally, 
\[
\begin{split}
n &\int_{U(\Gamma_{\rm in}, \frac{1}{n}) } \left[ \Ov{L_k(\vr)} - L_k(\vr) \right] \vu_b(\pig(x)) \cdot \frac{ x- \pig(x)}{|x - \pig(x)|}
\ \dx \\ &= n \int_{U(\Gamma_{\rm in}, \frac{1}{n}) \cap 
U(\partial \Gamma_{\rm in}, \frac{1}{n})  } \left[ \Ov{L_k(\vr)} - L_k(\vr) \right] \vu_b(\pig(x)) \cdot \frac{ x- \pig(x)}{|x - \pig(x)|}
\ \dx\\ &+ n \int_{U(\Gamma_{\rm in}, \frac{1}{n}) \cap 
U^c(\partial \Gamma_{\rm in}, \frac{1}{n})  } \left[ \Ov{L_k(\vr)} - L_k(\vr) \right] \vu_b(\pig(x)) \cdot \frac{ x- \pig(x)}{|x - \pig(x)|}
\ \dx.
\end{split}
\]
On one hand, 
as $\partial \Gamma_{\rm in}$ is a $C^1-$curve and all quantities under the integral are bounded, we get 
\[
n \int_{U(\Gamma_{\rm in}, \frac{1}{n}) \cap 
U(\partial \Gamma_{\rm in}, \frac{1}{n})  } \left[ \Ov{L_k(\vr)} - L_k(\vr) \right] \vu_b(\pig(x)) \cdot \frac{ x- \pig(x)}{|x - \pig(x)|}
\ \dx \approx \frac{1}{n}.
\]
On the other hand, by virtue of \eqref{h4} and convexity of $L_k$, 
\[
\begin{split}
n &\int_{U(\Gamma_{\rm in}, \frac{1}{n}) \cap 
U^c(\partial \Gamma_{\rm in}, \frac{1}{n})  } \left[ \Ov{L_k(\vr)} - L_k(\vr) \right] \vu_b(\pig(x)) \cdot \frac{ x- \pig(x)}{|x - \pig(x)|}
\ \dx \\&= 
- n \int_{U(\Gamma_{\rm in}, \frac{1}{n}) \cap 
U^c(\partial \Gamma_{\rm in}, \frac{1}{n})  } \left[ \Ov{L_k(\vr)} - L_k(\vr) \right] \vu_b(\pig(x)) \cdot 
\vc{n}(\pig(x)) \dx \leq 0.
\end{split}
\]
Going back to \eqref{a18} we deduce
\begin{equation} \label{a19} 
\frac{\D}{\dt} \intO{ \left[ \Ov{L_k(\vr)} - L_k(\vr) \right] } + 
\intO{ \left[ \Ov{T_k(\vr) \Div \vu } - T_k(\vr) \Div \vu \right] } \leq 0 
\end{equation}
in $\mathcal{D}'(R)$.

Now, we use \eqref{a16} to rewrite \eqref{a19} in the form
\begin{equation} \label{a20} 
\begin{split}
\frac{\D}{\dt} \intO{ \left[ \Ov{L_k(\vr)} - L_k(\vr) \right] } &+ d_0
\intO{ \left[ \Ov{p(\vr) T_k(\vr) } - \Ov{p(\vr)} \ \Ov{T_k(\vr)}  \right] } \\ &+ 
\intO{ \left( \Ov{T_k(\vr)} - T_k(\vr) \right) \Div \vu  }
\leq 0, 
\end{split}
\end{equation}
where we have set
\[
d_0 = \left(  \lambda + \frac{2 d - 2}{d}\mu \right)^{-1} > 0 \ (d = 3).
\]
Next, in view of \eqref{a16a}, 
\begin{equation*}
\intO{ \left( \Ov{T_k(\vr)} - T_k(\vr) \right) \Div \vu  } \to 0 \ \mbox{in}\ L^1(-M,M) 
\ \mbox{as}\ k \to \infty 
\end{equation*}
for any $M > 0$, see \cite[Chapter 6]{EF70}. Now, we are in the situation handled in \cite[Section 2]{EF53} or 
\cite[Chapter 6]{EF70}. In view of assumptions \eqref{h1} on the pressure function, arguing exactly as in \cite[Chapter 6, Section 6.6.2]{EF70}, we perform the limit 
$k \to \infty$ in \eqref{a20} to conclude
\begin{equation*}
\frac{\D }{\dt} \intO{ \left[ \Ov{\vr \log(\vr)} - \vr \log(\vr) \right] } + 
\Psi \left( \intO{ \left[ \Ov{\vr \log(\vr)} - \vr \log(\vr) \right] } \right) \leq 0 \ \mbox{in}\ \mathcal{D}'(R), 
\end{equation*}
where 
\[
\Psi \in C(R), \Psi(0) = 0,\ \Psi(Z) Z > 0 \ \mbox{for}\ Z \ne 0.
\]
As the function 
\[
t \mapsto \intO{ \left[ \Ov{\vr \log(\vr)} - \vr \log(\vr) \right](t, \cdot) }
\]
is non--negative and uniformly bounded on $R$, we conclude 
\[
\intO{ \left[ \Ov{\vr \log(\vr)} - \vr \log(\vr) \right](t, \cdot) } = 0 \ \mbox{for all}\ t \in R, 
\]
which yields the a.e. convergence of $\{ \vr_n \}_{n=1}^\infty$.

\medskip

\noindent{\bf Step 7 (strong compactness of the density):}

Revisiting {\bf Step 6} we observe that we have actually proved that
\[
\vr_n \to \vr \ \mbox{strongly in}\ L^q_{\rm loc}(R; L^q(Q)),\ 1 \leq q < \gamma.
\]
Seeing that, in view of Proposition \ref{wP1}, 
$
\vr_n, \ \vr \in C_{\rm loc}(R; L^1(Q)),
$
we show a stronger statement, namely 
\[
\vr_n \to \vr \in C_{\rm loc}(R; L^1(Q)).
\]
In particular,  we prove that
\[
\vr_n \to \vr \ \mbox{in}\  C([-M, M]; L^1(Q)) \ \mbox{for any}\ M > 0.
\]
To this end, consider the cut-off operators 
\[
T_k (\vr) = \min\{ \vr, k \}. 
\]
Now, pick $M>0$ and keep it fixed. As 
$
\vr_n$, $\vr$ are uniformly bounded in $L^\gamma(Q), 
$
we get 
\[ 
\| \vr_n - T_k(\vr_n) \|_{L^1(Q)} + \| \vr - T_k(\vr) \|_{L^1(Q)} \to 0 
\ \mbox{as}\ k \to \infty \ \mbox{uniformly for}\ t \in [-M,M].
\]
Consequently, it is enough to show 
\[
T_k(\vr_n) \to T_k(\vr) \ \mbox{in}\ C(-M; M; L^1(Q)) \ \mbox{as}\ n \to \infty 
\]
for any fixed $k$.

Similarly to the proof of Proposition \ref{wP1}, we consider 
\[
b(\vr_n) = T_k (\vr_n), \ T^2_k (\vr_n) 
\]
in the renormalized equation of continuity. 
In view of density compactness established in {\bf Step 6}, 
we may use the renormalized equation  of continuity to show
\begin{equation} \label{a31}
T_k(\vr_n) \to T_k(\vr) \ \mbox{in}\  C_{\rm weak}([-M, M]; L^2(Q)),\ 
T^2_k(\vr_n) \to T_k^2(\vr) \ \mbox{in}\  C_{\rm weak}([-M, M]; L^1(Q)).
\end{equation}
Seeing that 
\[
T^2_k(\vr_n) - T^2_k(\vr) = 2 T_k(\vr) \big(T_k(\vr_n) - T_k(\vr)\big) +  |T_k(\vr_n) - T_k(\vr)|^2,
\]
we observe it is enough to show 
\begin{equation} \label{a30}
T_k(\vr) \big(T_k(\vr_n) - T_k(\vr)\big) \to 0 \ \mbox{in}\ C_{\rm weak}([-M, M]; L^1(Q)).
\end{equation}
As the limit density is \emph{strongly continuous in} $L^1(Q)$, the image of the compact set $[-M,M]$ under the mapping 
\[
t \in [0,T] \mapsto T_k(\vr)(t, \cdot) \in L^q(Q) \ \mbox{is compact for arbitrary}\ 1 \leq q < \infty.
\]
In particular, the curve 
$
t \mapsto T_k(\vr)(t)
$
can be uniformly approximated by a piecewise constant function ranging in $L^{\gamma'}(Q)$. In particular, 
\[
T_k(\vr) \big(T_k(\vr_n) - T_k(\vr)\big) \approx \sum_{i} \mds{1}_{I_i}(t) w_i(x) \big(T_k(\vr_n) - T_k(\vr)\big), \quad \cup_{i} I_i = [-M,M], w_i \in L^{\gamma'}(Q);
\] 
whence \eqref{a30} follows from \eqref{a31}.

Theorem \ref{aP1} is now proved.

\subsection{Existence of trajectories having globally bounded energy} \label{ss:existence}

The goal of this subsection is to give sufficient (non--trivial) conditions on the boundary data $\vu_b$, which guarantee the existence of trajectories with globally bounded energy. For the sake of simplicity, we focus on the physically relevant case of a potential 
external force, 
\begin{equation} \label{POT}
\vc{g} = \Grad G,\ G \in C^1(\Ov{Q}).
\end{equation} 
Accordingly, the associated term in the energy inequality \eqref{w5} rewrites as 
\[
\intO{ \vr \vu \cdot \Grad G } =  \frac{\D}{\dt} \intO{ \vr G } + \int_{\Gamma_{\rm out}} \vr G \vu_b \cdot \vc{n} 
\ \D S_x  + \int_{\Gamma_{\rm in}} \vr_b G \vu_b \cdot \vc{n} \ \D S_x,  
\] 
and, similarly, 
\[
\intO{ \vr \vu \cdot \Grad |\vu_b|^2 } =  \frac{\D}{\dt} \intO{ \vr |\vu_b|^2 } + \int_{\Gamma_{\rm out}} \vr |\vu_b|^2 \vu_b \cdot \vc{n} 
\ \D S_x  + \int_{\Gamma_{\rm in}} \vr_b |\vu_b|^2 \vu_b \cdot \vc{n} \ \D S_x,  
\] 
Consequently, going back to the energy inequality \eqref{w5}, we deduce 
\begin{equation} \label{POT1}
\begin{split}
&\frac{\D}{\dt}  \intO{ \left[ E \left( \vr, \vu \Big| \vu_b \right) - \vr \left( G + \frac{1}{2} |\vu_b|^2 \right) \right] }   +
\intO{ \mathbb{S}(\Ds \vu) : \Ds \vu } \\ 
&\leq-  
\intO{ \left[ \vr \vu \otimes \vu + p(\vr) \mathbb{I} \right]  :  \Ds \vu_b }   -\intO{ \vr \Grad G \cdot 
\vu_b }  + K, 
\end{split}
\end{equation}
in $\mathcal{D}'(\tau_n, \infty)$, 
with a positive constant $K$ depending solely on the data $\vu_b, \vr_b, G$. We see immediately that the crucial term is the first integral on the right as it is proportional to the total energy. We therefore choose a ``cheap'' solution supposing 
that $\mathbb{D}_x \vu_b(x)$ is a positively semi--definite matrix. Moreover, we shall assume even more, namely that 
$\inf_Q \div \vu_b > 0$. Under these circumstances, we claim the following result. 

\begin{Proposition} [Bounded trajectories] \label{p:energy}

Let the driving force $\vc{g}$ satisfy \eqref{POT}, and let $\vu_b$ admit an extension satisfying 
\begin{equation} \label{ass:u_b_3}
\Ds\vu_b\,\geq\,0\,,\qquad\qquad \Div \vu_b\,\geq\, \alpha \,>\,0
\quad \mbox{ in }\ \oline Q.
\end{equation}
Suppose that $[\vr, \vm]$ is a finite energy weak solution of the Navier--Stokes system in $(\tau, \infty)$, 
$\tau \geq -\infty$  in the sense of Definition \ref{WD1}. 

Then there is a constant $\Ov{E}$, depending solely on the data, such that 
\begin{equation*}
\limsup_{t \to \infty} \intO{ E \left( \vr, \vu \Big| \vu_b \right)(t, \cdot) } \leq \Ov{E}.
\end{equation*}

\end{Proposition}

\begin{proof}

Under the hypothesis \eqref{ass:u_b_3}, we may rewrite the energy inequality \eqref{POT1} in the form 
\begin{equation} \label{POT4}
\frac{\D}{\dt}  \intO{ \left[ E \left( \vr, \vu \Big| \vu_b \right) - \vr \left( G + \frac{1}{2} |\vu_b|^2 \right) \right] }   +
\frac{\alpha}{2} \intO{ p(\vr) } + 
\intO{ \mathbb{S}(\Ds \vu) : \Ds \vu } \leq K 
\end{equation}
in $\mathcal{D}'(\tau, \infty)$, with a constant $K$, possibly different from its counterpart in \eqref{POT1}
but still depending only on the data. 

The first observation is that whenever 
\[
\frac{\alpha}{2} \intO{ p(\vr) } + 
\intO{ \mathbb{S}(\Ds \vu) : \Ds \vu } \leq K + 1 \ \mbox{for}\ \vr \in L^\gamma (Q), \ \vu \in W^{1,2}(Q; R^3),\ 
\vu|_{\partial \Omega} = \vu_b 
\]
then
\[ 
\intO{ \left[ E \left( \vr, \vu \Big| \vu_b \right) - \vr \left( G + \frac{1}{2} |\vu_b|^2 \right) \right] } \leq \Lambda(K).
\]

Going back to \eqref{POT4}, we deduce for any interval $[T, T + 1]$, $T > \tau$, the following dichotomy: 
{\bf (i)} either
 there exists $t_0 \in (T, T+1)$ such that  
\begin{equation} \label{POT5}
\intO{ \left[ E \left( \vr, \vu \Big| \vu_b \right) - \vr \left( G + \frac{1}{2} |\vu_b|^2 \right) \right] 
(t_0, \cdot) } \leq \Lambda(K), 
\end{equation}
$t_0$ a Lebesgue point of 
\[
\intO{ \left[ E \left( \vr, \vu \Big| \vu_b \right) - \vr \left( G + \frac{1}{2} |\vu_b|^2 \right) \right]}; 
\]
\tbf{(ii)} or one has
\[
\begin{split}
\intO{ \left[ E \left( \vr, \vu \Big| \vu_b \right) - \vr \left( G + \frac{1}{2} |\vu_b|^2 \right)((T + 1)-) \right] }\\
\leq \intO{ \left[ E \left( \vr, \vu \Big| \vu_b \right) - \vr \left( G + \frac{1}{2} |\vu_b|^2 \right)(T+) \right] } - 1.
\end{split}
\]

From the above, we deduce that there exists $t_0 > \tau$ such that \eqref{POT5} holds. Moreover,
as the modified energy 
\[
\intO{ \left[ E \left( \vr, \vu \Big| \vu_b \right) - \vr \left( G + \frac{1}{2} |\vu_b|^2 \right) \right]}
\]
is bounded from below, there is a sequence $t_n \to \infty$ such that \eqref{POT5} holds for $t_0 = t_n$ and 
$|t_n - t_{n+1}| \leq L$. 
Finally, revisiting the energy inequality \eqref{POT4} we deduce that 
\begin{equation*}
\intO{ \left[ E \left( \vr, \vu \Big| \vu_b \right) - \vr \left( G + \frac{1}{2} |\vu_b|^2 \right) \right] 
(t, \cdot) } \leq \Lambda(K) + LK \ \mbox{for any}\ t > t_0,
\end{equation*}
which yields the desired uniform bound on the energy.
\end{proof}

\section{Dynamical system generated by shifts on trajectories}

Consider the action of the shift operator on the trajectory space $\mathcal{T}$:
\[
S_\tau (r, \vc{w}) (\cdot) = (r, \vc{w})( \cdot+\tau), \ \tau \in R. \qquad \qquad
\]
It is easy to show that:
\begin{itemize} 
\item
$
S_\tau : \mathcal{T} \to \mathcal{T} \ \mbox{is a continuous linear operator for any}\ \tau \in R;
$
\item $(S_{\tau})_{\tau\in R}$ defines a group of operators on $\mathcal{T}$, i.e.,
$
S_0 = {\rm Id},\ S_{\tau_1 + \tau_2} = S_{\tau_1} \circ S_{\tau_2}\ \mbox{for any}\ \tau_{1},\tau_{2} \in R.
$
\end{itemize}
We say that a set $\mathcal{A}\subset\mathcal{T}$ is shift invariant provided $S_{\tau}\mathcal{A}\subset\mathcal{A}$ for any $\tau\in R$. Note that since the operators $(S_{\tau})_{\tau\in R}$ form a group, shift invariance  implies that $S_{\tau}\mathcal{A}=\mathcal{A}$ for any $\tau\in R$.

Recall from Section~\ref{ss:Traj} that $\mathcal{U}[\Ov{E}]$ denotes the set of all global solutions with energy uniformly bounded by $\Ov{E}$. Note, however, that $\mathcal{U}[\Ov{E}]$ may be empty,
unless some conditions are imposed on the boundary data (see Subsection \ref{ss:existence} above for details).
As a consequence of  Theorem \ref{aP1}, we obtain in particular the desired compactness of the set of solutions $\mathcal{U}[\Ov{E}]$.

\begin{Proposition} \label{TP1}
The set $\mathcal{U}[\Ov{E}]$ is a compact shift invariant subset of $\left[ \mathcal{T}; d_{\mathcal{T}} \right]$.
\end{Proposition}

\begin{proof}
The shift invariance is an immediate consequence of Definition~\ref{WD1}. It follows from Theorem~\ref{aP1} that $\mathcal{U}[\Ov{E}]$ is  compact.
\end{proof}

In view of the above, restriction of the translation semigroup $(S_{\tau})_{\tau\geq0}$ to the set $\mathcal{U}[\Ov{E}]$ together with its Borel $\sigma$-algebra defines a (topological) dynamical system on a compact Polish space. Then the notion of $\omega-$limit set $\omega[\vr,\vm]$ of  a finite energy weak solution $[\vr,\vm]$ in $(0,\infty)$ with uniformly bounded energy (see Section~\ref{ss:Traj}) corresponds to the usual definition of $\omega-$limit set within this dynamical system. Due to compactness of $\mathcal{U}[\Ov{E}]$, we obtain the following result.

\begin{Corollary} \label{TP2}
Let $[\vr, \vm]$ be a finite energy weak solution in $(0, \infty)$ satisfying 
\[
\intO{ E \left( \vr, \vm \Big| \vu_b \right) (t, \cdot) } \leq \Ov{E} \ \mbox{for all}\ t > 0.
\] 

Then the set $\omega[\vr, \vm]$ enjoys the following properties:
\begin{itemize}
\item
$
\omega[\vr, \vm] 
$
is non-empty;
\item 
$\omega[\vr, \vm]$ is a closed subset of $\mathcal{U}[\Ov{E}]$, in particular it is a compact subset of $[ \mathcal{T}; d_{\mathcal{T}}]$;

\item 
$\omega[\vr, \vm]$ is shift invariant;
\item $\omega[\vr, \vm]$ is connected.

\end{itemize}

\end{Corollary}

\begin{proof}
The fact that  the set $\omega[\vr, \vm]$ is a non-empty closed subset of 
$\mathcal{U}[\Ov{E}]$ follows from  Theorem~\ref{aP1}. The rest of the proof can be done by classical arguments from the theory of dynamical systems.
\end{proof}

\subsection{Non-wandering solutions}

With the above description of our model in the language of dynamical systems, we may deduce existence of solutions whose neighborhoods are revisited infinitely often.

\begin{Corollary}\label{cor:1}
Assume that $[\vr,\vm]$ is a global weak solution to the Navier-Stokes system, with $[\vr,\vm]\in\mc U(\oline E)$.
Then there exists a \emph{non-wandering} solution $[r,\vc{w}]\in\mathcal{U}[\Ov{E}]$, meaning,
for every $\varepsilon$-neighborhood $B_{\varepsilon}[r,\vc{w}]$ of $[r,\vc{w}]$ in $\mc T$ and every $T>0$, there exists $\tau>T$ so that
$B_{\varepsilon}[r,\vc{w}]\cap S_{\tau}B_{\varepsilon}[r,\vc{w}]\neq\emptyset$.
\end{Corollary}

\begin{proof}
The result is a simple consequence of Corollary~\ref{TP2} together with the fact that every point in $\omega[\vr,\vm]$ is non-wandering.
\end{proof}

Once we prove the existence of an invariant measure in the following section, we are able to deduce a stronger result than Corollary~\ref{cor:1}, namely, existence of the so-called recurrent solution, see Section~\ref{ss:R}.

\section{Invariant measures - stationary solutions on the space of trajectories}
\label{I}

We adopt the approach to statistical solutions via the theory of stochastic processes. Let 
\[
\{ \Omega, \mathfrak{B}, \mathcal{P} \}
\]
be a measurable space, with the $\sigma$-algebra of measurable sets $\mathfrak{B}$, and a (complete) probabi\-lity measure $\mathcal{P}$.
Let us first recall some basic definitions and facts.
\begin{itemize}
\item A \emph{(stochastic) process} $[\vr, \vm]$ ranging in a separable Banach space  $H$ is a mapping
\[
[\vr, \vm]: \Omega \times R \to H
\]
satisfying
\[
[\vr (t, \cdot), \vm(t, \cdot)] : \Omega \to H 
\]
is $\mathfrak{B}$-measurable for any $t \in R$.

\item A (stochastic) process $[\vr, \vm]$ ranging in $H$ is called \emph{measurable} provided the mapping
\[
[\vr, \vm]: \Omega \times R \to H
\]
is $\mathfrak{B}\otimes \mathfrak{B}(R)$-measurable where $\mathfrak{B}(R)$ denotes the Borel set in $R$.

\item For a measurable (stochastic) process $[\vr, \vm]$ ranging in $H$ the path
$$
t \in R \mapsto [\vr(t, \omega), \vm(t, \omega)]
$$
is measurable for a.a. $ \omega \in \Omega$.

\item A measurable (stochastic) process $[\vr, \vm]$ ranging in $H$ is called \emph{continuous} provided the path
\[
t \in R \mapsto [\vr(t, \omega), \vm(t, \omega)] \in C(R; H) \ \mbox{for a.a.}\ \omega \in \Omega.
\]
\end{itemize}

Alternatively, we may define a continuous (measurable, stochastic) process $[\vr, \vm]$ as a measurable mapping 
\[
[\vr, \vm]: \Omega \to C_{\rm loc}(R; H).
\]
Note that this is equivalent to the previous definition as the measurability of the process follows from the continuity of paths. Indeed, approximate $[\vr,\vm]$ by a process with piecewise constant paths, e.g.,
$$
[\vr,\vc{m}]_{n}(\omega,t)=[\vr,\vc{m}](\omega,t^{n}_{k}) \ \text{if}\ 
						t\in[t^{n}_{k},t^{n}_{k+1}),
$$
where $\{t^{n}_{k};k\in\mathbb{Z}\}$ is a suitable partition of $R$ with vanishing mesh size as $n\to\infty$. Then $[\vr,\vc{m}]_{n}$ is a measurable process  converging pointwise to $[\vr,\vm]$. Hence $[\vr,\vm]$ is measurable.

It is not always  convenient to work with stochastic processes as collections of random variables. Therefore, for the purposes of this paper, we make use of the last definition above, i.e., a stochastic process is a random variable taking values in the corresponding path space.

Let us now proceed with the definition of statistical solution and stationary statistical solution to the Navier--Stokes system.

\begin{Definition}[Statistical solution] \label{ID1}

A continuous stochastic process $[\vr, \vm]$ ranging in $L^{1}(Q) \times W^{-k,2}(Q; R^3)$
is a \emph{statistical solution} of the Navier--Stokes system if 
\[
[\vr, \vm] \ \mbox{is a finite energy weak solution of the Navier--Stokes system in}\ R \times Q 
\]
$\mathcal{P}-$a.s.

\end{Definition}

The probability space on which a statistical solution is defined does not play any particular role for us, since  the Navier--Stokes system is deterministic. Therefore,  we tacitly consider the probability space as part of the solution. Strictly speaking, a statistical solution is then a probability space together with a process satisfying the requirements of Definition~\ref{ID1}.
Statistical solutions are therefore conveniently described by their  laws. In particular, every statistical solution $[\vr,\vm]$ gives rise to a Borel probability measure on $\mathcal{T}$, i.e. its probability law, which is given as the pushforward measure of the mapping
$
[\vr,\vm]:\Omega\to\mathcal{T}.
$
We denote by $\mathfrak{P}(\mathcal{T})$ the set of Borel probability measures on $\mathcal{T}$ and for a continuous stochastic process $[\vr,\vm]$, we denote by $\mathcal{L}[\vr,\vm]$ its probability law, i.e., the Borel probability measure on $\mathcal{T}$ satisfying
$$
\mathcal{L}[\vr,\vm](B)=\mathcal{P}\left([\vr,\vm]\in B \right)\ \mbox{for any}\ B\subset\mathcal{T}\ \mbox{Borel.}
$$

Conversely,  statistical solutions in the sense of Definition~\ref{ID1} can be obtained from Borel probability measures supported in the set of finite energy weak solutions to the Navier--Stokes system.

\begin{Lemma}\label{lem:ID1}
Let $\nu\in\mathfrak{P}(\mathcal{T})$ be such that
$$
\supp\nu \subset \big\{[\vr, \vm]\,\big|\  [\vr, \vm] \ \mbox{\rm{is a finite energy weak solution of the Navier--Stokes system in}}\ R \times Q \big\}.
$$
Then there exists  a statistical solution $[\vr,\vm]$ such that $\mathcal{L}[\vr,\vm]=\nu$.
\end{Lemma}

\begin{proof}
Let $\Omega:=\mathcal{T}$ be equipped with its Borel $\sigma$-algebra and the probability measure  $\mathcal{P}:=\nu$. Then the canonical process
$$
[\vr,\vm](\omega,t):=\omega(t),\ \omega\in\Omega,t\in R,
$$
is a statistical solution in the sense of Definition~\ref{ID1} whose law under $\mathcal{P}$ is $\nu$.
\end{proof}

\begin{Remark}[Connection to Markov statistical solutions from \cite{FanFei}]\label{rem:FF}
Another notion of statistical solution was introduced in \cite{FanFei}.
There, a statistical solution is family of Markov operators $\{M_{t}\}_{t\geq 0}$ acting on the set of probability measures $\mathfrak{P}(\mathcal{D})$ where $\mathcal{D}$ is the corresponding
set of input data, i.e., initial and boundary conditions. Since uniqueness for the Navier--Stokes system is an open problem, the two notions of solution are not equivalent. More precisely,
the Markov statistical solution in the sense of  \cite{FanFei} was constructed by means of a selection procedure in the spirit of Krylov~\cite{KrylNV}, see also
\cite{BrFeHo2018C,BreFeiHof19C,BreFeiHof19} for  related results on Markov and semiflow  selections. Every such solution gives rise to a statistical solution in the sense of Definition~\ref{ID1}
through the procedure by Da~Prato and Zabczyk~\cite[Section 2.2]{DPZ96}. On the other hand,  not every statistical solution in the sense of Definition~\ref{ID1} satisfies the corresponding Markov
property as this is a priori a property of systems with uniqueness, or alternatively it stems from a suitable selection.
\end{Remark}

As the next step, we define stationary statistical solution which is the main object of interest of this paper. The notion is similar to the corresponding SPDE setting, see \cite{BrFeHo2017}. In the sequel, we use the same notation $\mathcal{L}(X)$ for the probability law of a general random variable $X:\Omega\to \mathcal{X}$ taking values in the target space $\mathcal{X}$. In our setting $\mathcal{X}$ will always be a Polish space.

\begin{Definition} [Stationary statistical solution] \label{ID2}

A statistical solution $[\vr, \vm]$ of the Navier--Stokes system is \emph{stationary} if it is a stationary process, meaning the joint laws 
\[
\mathcal{L}\left( [\vr(t_1, \cdot), \vm(t_1, \cdot)], \dots, [\vr(t_n, \cdot), \vm(t_n, \cdot)] \right),
\]
\[  
\mathcal{L}\left( [\vr(t_1 + \tau , \cdot), \vm(t_1 + \tau , \cdot)], \dots, [\vr(t_n + \tau, \cdot), \vm(t_n + \tau, \cdot)] \right)
\]
coincide for any $t_1 < \dots < t_n$ and any $\tau \in R$. Equivalently,  we may say that 
the law on trajectories
$
\mathcal{L} [\vr, \vm] \in \mathfrak{P} (\mathcal{T}) 
$
is shift invariant, meaning, $\mathcal{L} [\vr, \vm]=\mathcal{L} (S_{\tau}[\vr, \vm])$ for all $\tau\in R$.
\end{Definition}

Note that  $\mathcal{L}\left( [\vr(t_1, \cdot), \vm(t_1, \cdot)], \dots, [\vr(t_n, \cdot), \vm(t_n, \cdot)] \right)$ is the pushforward measure of the random vector
$$
\left( [\vr(t_1, \cdot), \vm(t_1, \cdot)], \dots, [\vr(t_n, \cdot), \vm(t_n, \cdot)] \right)
$$
which takes values in $[L^{1}(Q)\times W^{-k,2}(Q;R)]^{n}$, and it is necessary to include an arbitrary number of times $t_{1},\dots,t_{n}$ in the definition. On the other hand, $\mathcal{L}[\vr,\vm]$ is the pushforward measure of the whole process $[\vr,\vm]:\Omega\to\mathcal{T}$.
The equivalence of the two above formulations of stationarity was proved in \cite[Section~2.11]{BrFeHobook}, see in particular Lemma 2.11.7 and Lemma 2.11.5.

We also remark  that, in the language of dynamical systems, a stationary statistical solution $[\vr,\vm]$ generates a measure--preserving dynamical system through
$$
(\mathcal{T},\mathfrak{B}[\mathcal{T}],(S_{\tau})_{\tau\in R},\mathcal{L}[\vr,\vm]).
$$
In other words, the probability measure $\mathcal{L}[\vr,\vm]$ is invariant under the transformations $(S_{\tau})_{\tau\in R}$.

\subsection{Existence of stationary statistical solutions} \label{ss:ex-stat}

The next result asserts the existence of at least one stationary statistical solution provided the Navier--Stokes system admits 
a global solution with bounded energy. 

\begin{Theorem}[Existence of stationary statistical solution] \label{IT1}

Let  $\mathcal{A} \subset \mathcal{U}[\Ov{E}] \subset \mathcal{T}$ be a non-empty shift invariant set of trajectories.

Then there exists a stationary statistical solution $[\vr, \vm]$ such that
\[
[\vr, \vm] \in \Ov{\mathcal{A}} \subseteq \mathcal{U}[\Ov{E}] \ \mbox{a.s.}
\]
In particular, for any finite energy weak solution $[\tvr, \tvm]$ in $(0, \infty)$ such that
\[
\limsup_{t \to \infty} 
\intO{ E \left( \tvr, \tvm \Big| \vu_b \right)(t, \cdot) } < \infty  ,
\]
there exists a stationary statistical solution $[\vr, \vm]$ 
such that 
\[
[\vr, \vm] \in \omega[ \tvr, \tvm] \ \mbox{a.s.}
\]

\end{Theorem}

\begin{proof}
The second statement is a consequence of the first one with $\mathcal{A}=\omega[\tvr,\tvm]$. Indeed, according to Corollary~\ref{TP2}, the $\omega$--limit set $\omega[\tvr,\tvm]$ is non-empty,
closed and shift invariant.
 In order to prove the first statement, we observe that in view of Lemma~\ref{lem:ID1}, it is sufficient to find a shift invariant probability measure $\nu\in\mathfrak{P}(\mathcal{T})$ which is supported by $\Ov{\mathcal{A}}$.
Due to the compactness of the set $\Ov{\mathcal{A}}$, the existence of such a measure follows immediately from Markov--Kakutani's fixed point theorem \cite[Theorem V.20]{ReSi80}.
Nevertheless, let us present a  constructive proof based on the so-called Krylov--Bogolyubov argument, which gives more insight into the properties of the measure.

By virtue of Proposition \ref{aP1}, the 
set $\Ov{\mathcal{A}} \subseteq \mathcal{U}[\Ov{E}]$ is a compact shift invariant subset of the Polish space $\mathcal{T}$. Thus we may define
a family of probability measures 
\begin{equation}\label{eq:I1}
T \mapsto \nu_T \equiv \frac{1}{T} \int_0^T \delta_{S_t(\tvr, \tvm)} \dt  \in \mathfrak{P}(\mathcal{T}) 
\end{equation}
for an arbitrary $(\tvr, \tvm) \in \mathcal{A}$. Due to shift invariance of $\mathcal{A}$, the approximate measures \eqref{eq:I1} are all supported in the compact subset $\Ov{\mathcal{A}}$ of the Polish space $\mathcal{T}$. Hence the family is tight and by Prokhorov's theorem there is a subsequence $T_{n}\to\infty$ and a probability measure $\nu\in\mathfrak{P}(\mathcal{T})$ such that
\[ 
\nu_{T_n}\,\equiv\frac{1}{T_n} \int_0^{T_n} \delta_{S_t(\tvr, \tvm)} \dt \to \nu 
\ \mbox{narrowly in}\ 
\mathfrak{P}(\mathcal{T}),
\]
meaning 
\begin{equation} \label{I1}
\frac{1}{T_n} \int_0^{T_n} G (\tvr(\cdot + t) , \tvm(\cdot + t)) \dt\to \int_{\mathcal{T}} G(r, \vc{w}) \D \nu (r, \vc{w})
\end{equation}
for any $G \in BC(\mathcal{T})$.

Moreover, since the approximate sequence was supported in $\Ov{\mc A}$, also the limit measure $\nu$ is supported on $\Ov{\mc A}$. Indeed, the narrow convergence  $\nu_{T_{n}}\to \nu$ is equivalent to
$$
\limsup_{n\to\infty}\nu_{T_{n}}(\mathcal C)\leq \nu(\mathcal C)\ \mbox{for any closed set}\ \mathcal C\subset\mathcal T.
$$
And taking $\mathcal{C}=\Ov{\mathcal{A}}$ implies
$
\nu(\Ov{\mathcal{A}})=1.
$

Let us show that  $\nu$ is shift invariant. To this end, let $G\in BC(\mathcal{T})$ and $\tau\in R$. According to the continuity of the translation operator $S_{\tau}:\mathcal{T}\to \mathcal{T}$ we deduce that $G\circ S_{\tau}\in BC(\mathcal{T})$. Hence
$$
\int_{\mathcal{T}} G\circ S_{\tau}(r, \vc{w}) \D \nu (r, \vc{w})= \lim_{n\to\infty}\frac{1}{T_{n}}\int_{0}^{T_{n}} G(\tilde\vr(\cdot+\tau+t), \tilde\vm(\cdot+\tau+t)) \,\D t
$$
$$
=\lim_{n\to\infty}\frac{1}{T_{n}}\int_{\tau}^{\tau+T_{n}} G(\tilde\vr(\cdot+s), \tilde\vm(\cdot+s)) \,\D s
$$
$$
=\lim_{n\to\infty}\frac{1}{T_{n}}\int_{0}^{T_{n}} G(\tilde\vr(\cdot+s), \tilde\vm(\cdot+s)) \,\D s
$$
$$
+\lim_{n\to\infty}\frac{1}{T_{n}}\left[\int_{T_{n}}^{\tau+T_{n}} G(\tilde\vr(\cdot+s), \tilde\vm(\cdot+s)) \,\D s-\int_{0}^{\tau} G(\tilde\vr(\cdot+s), \tilde\vm(\cdot+s)) \,\D s\right].
$$
By boundedness of $G$, the last line above vanishes. Therefore in view of \eqref{I1} we deduce that
$$
\int_{\mathcal{T}} G\circ S_{\tau}(r, \vc{w}) \D \nu (r, \vc{w})=\int_{\mathcal{T}}G(r,\vc{w}) \,\D \nu(r,\vc{w})
$$
so $\nu$ is  shift invariant.

As suggested at the beginning of the proof, we may conclude by the application of Lemma~\ref{lem:ID1}.
\end{proof}

\subsection{Continuity of the total energy} \label{ss:cont-en}

The total energy 
\[
\intO{ E \left(\vr, \vm \Big| \vu_b \right) }
\]
is a Borel measurable function on the Banach space $L^1(Q) \times W^{-k,2}(Q; R^3)$ and as such an observable quantity of any 
statistical solution.
Similarly, for a statistical solution $[\vr, \vm]$ we introduce the quantities
\[
\begin{split}
\mathcal{E}(t-) &= \lim_{\tau \to 0+} \frac{1}{\tau} {\int_{t - \tau}^t }\left( \intO{ E \left(\vr(s,\cdot), \vm(s,\cdot) \Big| \vu_b \right) } \right) \D s
\\
\mathcal{E}(t+) &= \lim_{\tau \to 0+} \frac{1}{\tau} {\int_t^{t + \tau}} \left( \intO{ E \left(\vr(s,\cdot), \vm(s,\cdot) \Big| \vu_b \right) } \right) \D s 
\end{split}
\]
and 
\[
\mathcal{E}(t) = \intO{ E \left(\vr(t, \cdot), \vm(t, \cdot) \Big| \vu_b \right) }.
\]
As $[\vr, \vm]$ satisfies the energy inequality \eqref{w5}, all the above quantities are a.s. well defined for any $t \in R$, and according to weak lower semicontinuity of the total energy, it holds a.s.
\begin{equation}\label{eq:E1}
\mathcal{E}(t)\leq\liminf_{s\to t}\mathcal{E}(s) \ \mbox{for all}\ t \in R,
\end{equation}
hence a.s.
\begin{equation} \label{s1}
\mathcal{E}(t) \leq \mathcal{E}(t+) ,\ \mathcal{E}(t) \leq \mathcal{E}(t-) \ \mbox{for all}\ t \in R. 
\end{equation}
Note that here the corresponding set of full probability does not depend on $t$ as it is precisely the set of full probability where $[\vr,\vm]$ solves the Navier--Stokes system. 
 We recall that equality in both inequalities in \eqref{s1} implies continuity of the solution $[\vr, \vm]$ at the time $t$ in the strong topology of the space $L^\gamma(Q) \times L^{\frac{2\gamma}{\gamma + 1}}(Q; R^{3})$. The following result shows that for a stationary solution 
\eqref{s1} this holds a.s. for any fixed $t \in R$.

\begin{Theorem}[Strong continuity of stationary statistical solution] \label{sL1}

Let $[\vr, \vm]$ be a stationary statistical solution of the Navier--Stokes system. 

Then for any $t \in R$ we have a.s. 
\[
\mathcal{E}(t) = \mathcal{E}(t+) = \mathcal{E}(t-).
\]

\end{Theorem}

\begin{proof}
Let us first assume that $\Exp[\mathcal{E}(0)]<\infty$ which by stationarity implies  $\Exp[\mathcal{E}(t)]<\infty$ for all $t\in R$.
Given $t \in R$, consider the random variable 
\[
\frac{1}{2} \left[ \mathcal{E}(t+) + \mathcal{E}(t-) \right] - \mathcal{E}(t) 
= \lim_{\tau \to 0+} \frac{1}{2 \tau} \int_{t - \tau}^{t + \tau} \left( \intO{ E \left(\vr, \vm \Big| \vu_b \right) } \right) \D s 
- \mathcal{E}(t) \geq 0 \ \mbox{a.s.}
\]
Passing to expectations and using stationarity, we deduce
\[ 
\Exp \left[ \frac{1}{2} \left[ \mathcal{E}(t+) + \mathcal{E}(t-) \right] - \mathcal{E}(t) \right] 
= \lim_{\tau \to 0+} \frac{1}{2 \tau} \int_{t - \tau}^{t + \tau} 
\Exp \left[  \intO{ E \left(\vr, \vm \Big| \vu_b \right) } \right]  \D s 
- \Exp [ \mathcal{E}(t) ] = 0.
\]
This implies that
$$
\frac{1}{2} \left[ \mathcal{E}(t+) + \mathcal{E}(t-) \right] = \mathcal{E}(t)  \ \mbox{a.s.}
$$
and the claim follows due to \eqref{s1}.

If the total energy is not  in $L^{1}(\Omega)$, we include a suitable cut-off function. Namely, let $\beta:[0,\infty)\to[0,\infty)$ be a strictly increasing, smooth and bounded function such that $\beta(0)=0$. Then \eqref{eq:E1} implies
$$
\beta(\mathcal{E}(t))\leq \liminf_{s\to t}\beta(\mathcal{E}(s)) \ \mbox{for all}\ t \in R,
$$
and defining analogously
$$
\beta(\mathcal{E}(t-) )= \lim_{\tau \to 0+} \frac{1}{\tau} {\int_{t - \tau}^t }\beta( \mathcal{E}(s))\, \D s,\qquad
\beta(\mathcal{E}(t+)) = \lim_{\tau \to 0+} \frac{1}{\tau} {\int_t^{t + \tau}} \beta( \mathcal{E}(s))\,\D s 
$$
we deduce that a.s.
\begin{equation*}
\beta(\mathcal{E}(t)) \leq \beta(\mathcal{E}(t+)) ,\ \beta(\mathcal{E}(t)) \leq \beta(\mathcal{E}(t-)) \ \mbox{for all}\ t \in R. 
\end{equation*}
Hence, repeating the above with $\mathcal{E}$ replaced by $\beta(\mathcal{E})$ completes the proof.
\end{proof}

\begin{Remark}
Note that the set of zero probability in the statement of Theorem~\ref{sL1} possibly depends on $t$. In other words, while  at every time $t$ the energy is a.s. continuous, it does not follow that a.e. trajectory is continuous.
\end{Remark}

\subsection{Recurrent solutions}\label{ss:R}

With the existence of an invariant measure at hand, we are able to deduce existence of a recurrent  solution.

\begin{Corollary}\label{cor:rec}
Assume there exists a solution $[\vr,\vm]$ with uniformly bounded energy. Then there exists a \emph{recurrent} solution $[r,\vc{w}]\in\omega[\vr,\vm]$, meaning, $[r,\vc{w}]$ belongs to its own $\omega$--limit set, i.e., $[r,\vc{w}]\in\omega[r,\vc{w}].$ In particular, for  every $\varepsilon>0$ there exists  $T>0$ such that
$$
d_{\mathcal{T}}\Big[[r,\vc{w}](\cdot+T);[r,\vc{w}](\cdot)\Big]<\varepsilon.
$$
\end{Corollary}

\begin{proof}
A simple reduction to the discrete time setting permits to apply Poincar\'e's recurrence theorem \cite[Theorem I.2.3]{Ma87}. It yields  that for every invariant measure $\nu$ the set of recurrent points is of full measure. Since by Theorem~\ref{IT1} there exists an invariant measure supported on the $\omega$--limit set $\omega[\vr,\vm]$ whenever $[\vr,\vm]$ is a finite energy weak solution with \eqref{T1}, it follows that every $\omega[\vr,\vm]$ contains at least one recurrent solution $[r,\vc{w}]$.
\end{proof}

\section{Ergodic theory, application of Birkhoff--Khinchin Theorem}
\label{A}

At this level, it is convenient to work in the abstract setting. Recall that a stationary statistical solution $[\vr, \vm]$ can be identified with a stationary 
process ranging in $H$ with continuous paths, or, equivalently with its law $\mathcal{L}[\vr, \vm]$, 
a shift invariant Borel probability measure on the space of trajectories $[\mathcal{T}, d_{\mathcal{T}}]$.

\subsection{Birkhoff--Khinchin theorem} \label{ss:BK}

For a stationary statistical solution $[\vr, \vm]$, we consider the ergodic limit 
\[
\lim_{T \to \infty} \frac{1}{T} \int_0^T F(\vr(t, \cdot), \vm(t, \cdot)) \dt, 
\]
where $F: H \to R$ is a bounded Borel measurable function. The following result is a straightforward application of Birkhoff--Khinchin 
ergodic theorem.

\begin{Theorem}[Ergodic property] \label{AT1}
Let $[\vr, \vm]$ be a stationary statistical solution of the Navier--Stokes system, and let $F$ be a bounded Borel measurable function on $H$.

Then there is an observable\footnote{A measurable function $\Ov{F}:\Omega\to R$ is called observable.} function $\Ov{F}$ such that
\[
\frac{1}{T} \int_0^T F(\vr(t, \cdot), \vm(t, \cdot)) \dt \to \Ov{F} \ \mbox{as}\ T \to \infty\ \mbox{a.s.}
\]
In particular,
\[
\Exp \left[\left| \frac{1}{T} \int_0^T F(\vr(t, \cdot), \vm(t, \cdot)) \dt - \Ov{F} \right|\right] \to 0 
\ \mbox{as}\ T \to \infty. 
\]

\end{Theorem}

\begin{proof}
The first statement concerning the a.s. convergence of the ergodic averages is the  version of  Birkhoff--Khinchin's ergodic theorem for stochastic processes proved by Kolmogorov
\cite[Chapter~39]{Kolmo}.
Next, since
$$
\Exp\left[\left|\frac{1}{T}\int_{0}^{T}F(\vr(t,\cdot),\vm(t,\cdot))\right|\right]\dt \leq \Exp\left[\left| F(\vr(0,\cdot),\vm(0,\cdot))\right|\right]\leq \|F\|_{L^{\infty}(H)},
$$
dominated convergence and stationarity imply the convergence in $L^1(\Omega)$.
\end{proof}

\begin{Remark} \label{Ri5}

As we have seen in Theorem \ref{IT1}, there are stationary solutions supported by $\mathcal{U}[\Ov{E}]$. The total energy 
\[
\mathcal{E} = \intO{ E \left( \vr, \vm \Big| \vu_b \right) } 
\]
is then a bounded Borel measurable function to which the conclusion of Theorem \ref{AT1} applies.

\end{Remark}

Theorem \ref{AT1} can be extended to  more general functions $F$. We proceed through an auxiliary  lemma, whose main idea is due to Kolmogorov \cite[Chapter 39]{Kolmo}.

\begin{Lemma} \label{ALL1}

Let $U: \Omega\times [0, \infty)  \to R$ be a  measurable stationary stochastic process such that 
\[
{\Exp[|U(0)|]<\infty\ \mbox{and}}\ U \in L^1_{\rm loc}[0, \infty) \ a.s.
\]

Then 
\[
\lim_{T \to \infty} \frac{1}{T} \int_0^T U(t) \dt \to \Ov{U} \ \mbox{a.s. and in} \  L^{1}(\Omega).  
\]

\end{Lemma}

\begin{proof}

Splitting $U$ into its positive and negative part, we observe that it is enough to show the result for non-negative $U$. 
The integral averages 
\[
U_n = \int_n^{n+1} U(t) \ \dt, \ n=0,1, \dots 
\]
are well defined and represent a stationary discrete process in the sense of Krylov \cite[Chapter~4, Section 6, Definition 1]{Krylovstoch}. Thus applying the discrete version of Birkhoff--Khinchin's Theorem \cite[Chapter 4, Section 6, Theorem 11]{Krylovstoch}  we obtain the existence of $\Ov{U}:\Omega\to R$ such that for  $N\to\infty$ where $N$ takes only discrete values $N=1,2,\dots,$ 
\[
\frac{1}{N} \sum_{n = 1}^{N-1} U_n = \frac{1}{N} \int_0^N U(t) \ \dt \to \Ov{U} \ \mbox{a.s. and in} \  L^{1}(\Omega).
\]

Finally, denoting by $[T]$ the largest integer less than or equal to $T$ and using non-negativity of $U$, we get
\[
\Ov{U} \leftarrow
\frac{[T]}{T} \frac{1}{[T]} \int_0^{[T]} U(t) \ \dt
\leq \frac{1}{T} \int_0^T U(t) \ \dt \leq \frac{[T]+1}{T} \frac{1}{[T] + 1} \int_0^{[T] + 1} U(t) \ \dt 
\rightarrow \Ov{U},
\]
in the limit for for $T\to\infty$.
\end{proof}

With this result in hand, we are able to prove the generalization of Theorem~\ref{AT1}.

\begin{Theorem} \label{ACC1}

Let $[\vr, \vm]$ by a stationary statistical solution of the Navier--Stokes system and $F: H \to R$ Borel measurable such that
\[
\Exp{\Big[|F(\vr(0, \cdot), \vm(0, \cdot)) |\Big] } < \infty.
\]

Then there is an observable function $\Ov{F}$ such that
\[
\frac{1}{T} \int_0^T F(\vr(t,\cdot), \vm(t,\cdot)) \ \dt \to \Ov{F} \ \mbox{as}\ T\to\infty\ \mbox{a.s. and in}\ L^{1}(\Omega).
\]

\end{Theorem}

\begin{proof}

The proof follows from Lemma \ref{ALL1} as long as we observe that 
$t\mapsto F(\vr(t,\cdot), \vm(t,\cdot))$ is locally integrable a.s. To see this, write 
\[
\Exp{ \left[ \int_{-M}^M |F(\vr(t,\cdot), \vm(t,\cdot))| \ \dt \right]} = 
\int_{-M}^M \Exp\big[{ |F(\vr(t,\cdot), \vm(t,\cdot))| }\big] \dt \leq 2 M \Exp\big[|F(\vr(0,\cdot), \vm(0,\cdot))|\big] < \infty.
\]
Hence, we deduce that 
\[
\mathcal{P} \left\{ \int_{-M}^M |F(\vr(t,\cdot), \vm(t,\cdot))| \ \dt < \infty \right\} = 1 \ \mbox{for any}\ M = 1,2,\dots ,
\] 
and consequently 
\[
\mathcal{P} \Big\{ F(\vr(t,\cdot), \vm(t,\cdot)) \in L^1_{\rm loc}(R) \Big\} = 
\mathcal{P} \left\{\bigcap_{M =1,2,\dots} \left\{ \int_{-M}^M |F(\vr(t,\cdot), \vm(t,\cdot))| \ \dt < \infty \right\}\right\} = 1,
\] 
which completes the proof.
\end{proof}

\begin{Remark} \label{RCC1}

Beware that the set of probability zero in Theorem \ref{ACC1} in general depends on $F$.
\end{Remark}

\section{Ergodic stationary solutions}

Of essential interest in turbulence theory is the so--called ergodic hypothesis, which is usually assumed by physicists and engineers, based on some empirical evidences. Roughly speaking, it assures that time averages along solution trajectories coincide with ensemble averages   with respect to a putative probability measure. This measure can then be shown to be invariant for the dynamics.

In our context, since the energetically open Navier--Stokes system may admit multiple equilibrium states, hence multiple invariant measures, the limit in the ergodic hypothesis necessarily depends on the chosen trajectory $[\vr,\vm]$. For instance,  in the simplest case when $\vu_{b}\equiv 0$ and $\bfg\equiv 0$, then there exists infinitely many deterministic stationary (i.e. time independent) solutions,
given by $\vm\equiv 0$ and $\vr\equiv \vr_{0}$, for every  $\vr_{0}\in[0,\infty)$. Every such solution generates a stationary statistical solution in the sense of Definition~\ref{ID2}
and its law $\mathcal{L}[\vr,\vm]$ is given by the Dirac mass supported on $[\vr,\vm]$.

Therefore, we formulate the ergodic hypothesis as follows:  for an  finite energy weak solution $[\vr,\vm]$ with globally bounded energy, the limit of time averages
\begin{equation}\label{eq:EE1}
\lim_{T\to\infty}\frac{1}{T}\int_{0}^{T}F(\vr(t,\cdot),\vm(t,\cdot))\dt 
\end{equation}
exists for any \emph{bounded continuous} function $F:H\to R.$ Note that the quantity in \eqref{eq:EE1} then defines a probability measure on $H$ which is invariant  for the dynamics.

It will be seen below that the convergence of the ergodic averages \eqref{eq:EE1} corresponds to the ergodicity property of the corresponding invariant measure. This motivates our next definition.

\begin{Definition} [Ergodic stationary statistical solution] \label{AD1}

A stationary statistical solution $[\vr, \vm]$, or its law $\mathcal{L}[\vr, \vm]$ on $\mathcal{T}$, is called \emph{ergodic}, if the $\sigma$-field 
of shift invariant sets is trivial, specifically, 
\[
\mathcal{L}[\vr, \vm] (B) = 1 \ \mbox{or}  \ \mathcal{L}[\vr, \vm](B) = 0 \ \mbox{for any shift invariant Borel set}\ B \in \mathfrak{B}[\mathcal{T}].
\]

\end{Definition}

As a direct consequence of Theorem \ref{AT1} and Theorem~\ref{ACC1}, we get the following result for ergodic stationary statistical solutions.

\begin{Theorem}[Ergodicity] \label{AC1}
Let $[\vr, \vm]$ be an ergodic stationary statistical solution of the Navier--Stokes system. Let $F:H\to R$ be a Borel measurable function such that
$$
\Exp\Big[|F(\vr(0,\cdot),\vm(0,\cdot))|\Big]<\infty.
$$

Then 
\[
\frac{1}{T} \int_0^T F(\vr(t, \cdot), \vm(t, \cdot)) \dt \to \Exp [F(\vr(0,\cdot),\vm(0,\cdot))] \ \mbox{as}\ T \to \infty\ \mbox{a.s.}
\]  

\end{Theorem} 

\begin{proof}
Let us first prove the claim for the case of the the canonical process on $\mathcal{T}$, i.e.,
$$
(\Omega,\mathfrak{B},\mathcal{P}):=(\mathcal{T},\mathfrak{B}[\mathcal{T}],\mathcal{L}[\vr,\vm]),
\quad [\vr,\vm](t,\omega)=\omega(t),\ \mbox{for}\ t\in R,\  \omega\in\mathcal{T}.
$$
Let us show that the limit of the ergodic averages given by Theorem~\ref{AT1} satisfies
\begin{equation}\label{eq:EE2}
\Ov{F}=\Exp\left[F(\vr(0,\cdot),\vm(0,\cdot))\Big|\mathfrak{B}_{S}[\mathcal{T}]\right],
\end{equation}
where $\mathfrak{B}_{S}[\mathcal{T}]$ is the $\sigma$-algebra of shift invariant sets in $ \mathcal{T}$.
To this end, we first observe that as an immediate consequence of Theorem~\ref{ACC1} together with the shift invariance of $\mathcal{L}[\vr,\vm]$, we obtain
$$
\Ov{F}\in L^{1}(\Omega),\ 
\Exp\left[ \Ov{F}\right]=\Exp\left[ F(\vr(0,\cdot),\vm(0,\cdot))\right].
$$
Next, we prove that $\Ov{F}$ is invariant with respect to time shifts $S_{\tau}$, $\tau\in R$. Indeed, if $F$ is bounded then
$$
\Ov{F}\circ S_{\tau}=\lim_{T\to\infty}\frac{1}{T}\int_{0}^{T}F(S_{\tau}[\vr,\vm](t,\cdot))\,\D t = \lim_{T\to\infty}\frac{1}{T}\int_{0}^{T}F(\vr(t,\cdot),\vm(t,\cdot))\,\D t =\Ov{F}
$$
by the same argument as when proving the shift invariance of the limiting measure $\nu$ in the proof of Theorem~\ref{IT1}.
If $F$ is not bounded but only $F\in L^{1}(H,\mathcal{L}[\vr(0,\cdot),\vm(0,\cdot)])$ by the assumption, then $F$ can be approximated in
$L^{1}(H,\mathcal{L}[\vr(0,\cdot),\vm(0,\cdot)])$ by bounded functions $F_{n}$. Let $\Ov{F}_{n}$ denote the associated limits of ergodic averages. Then we have
$$
\Exp\left[\left|\Ov{F}_{n}-\Ov{F}\right|\right]\leq \Exp\left[\left|\Ov{F}_{n}-\frac{1}{T}\int_{0}^{T}F_{n}(\vr(t,\cdot), \vm(t,\cdot))\,\D t\right|\right]
$$
$$
+\Exp\left[\left|\frac{1}{T}\int_{0}^{T}F_{n}(\vr(t,\cdot), \vm(t,\cdot))\,\D t-\frac{1}{T}\int_{0}^{T}F(\vr(t,\cdot), \vm(t,\cdot))\,\D t\right|\right]
$$
$$
+\Exp\left[\left|\frac{1}{T}\int_{0}^{T}F(\vr(t,\cdot), \vm(t,\cdot))\,\D t-\Ov{F}\right|\right].
$$
For every $n$, the first term vanishes as $T\to\infty$, by stationarity the second term vanishes as $n\to\infty$ uniformly in $T$, and the last term is small for $T\to\infty$. Therefore, we deduce that $\Ov{F}_{n}\to\Ov{F}$ in $L^{1}(\Omega)$. Since all approximations $\Ov{F}_{n}$ are shift invariant, the same remains valid for $\Ov{F}$.

As a consequence of the shift invariance, 
$\Ov{F}$ is measurable with respect to  $\mathfrak{B}_S [\mathcal{T}]$.
Moreover, if $S_{\tau}A=A$ for all $\tau\in R$ then 
\begin{equation*}
\Exp\left[ \Ov{F} {\mds 1}_{A}\right]=\lim_{T\to\infty}\frac{1}{T}\int_{0}^{T}\Exp\left[ {F}(\vr(t,\cdot),\vm(t,\cdot)) {\mds 1}_{A}\right] \dt=\Exp \left[ {F}(\vr(0,\cdot),\vm(0,\cdot)) {\mds 1}_{A}\right].
\end{equation*}
Thus, \eqref{eq:EE2} follows.
By the assumption of ergodicity of the measure $\mathcal{L}[\vr,\vm]$, all the sets in $\mathfrak{B}_{S}[\mathcal{T}]$ are of zero or full measure. Hence we  deduce that
$$
\Exp\left[F(\vr(0,\cdot),\vm(0,\cdot))\Big|\mathfrak{B}_{S}[\mathcal{T}]\right]=\Exp[F(\vr(0,\cdot),\vm(0,\cdot))],
$$
which proves the claim.

Assume now that the stationary statistical solution $[\vr,\vm]$ is defined on some general probability space $(\Omega,\mathfrak{B},\mathcal{P})$. Then the law of the ergodic average
\begin{equation}\label{eq:EE3}
\frac{1}{T} \int_0^T F(\vr(t, \cdot), \vm(t, \cdot)) \dt
\end{equation}
under $\mathcal{P}$
coincides with the law of the ergodic average of the canonical process $[r,\vc{w}]$ under $\mathcal{L}[\vr,\vm]$. According to the previous part of the proof, it therefore follows that the ergodic averages \eqref{eq:EE3} converge in law to the constant given by
\begin{equation}\label{eq:EE4a}
\int_{\mathcal{T}}F(r(0,\cdot),\vc{w}(0,\cdot))\,\D\mathcal{L}[\vr,\vm](r,\vc{w})=\Exp[F(\vr(0,\cdot),\vm(0,\cdot))].
\end{equation}
Therefore, the a.s. limit of \eqref{eq:EE3} is also given by \eqref{eq:EE4a} and the proof is complete.
\end{proof}

\subsection{Existence of ergodic solutions and ergodic decomposition}
\label{ERG}

As a consequence of Theorem~\ref{AC1}, the ergodic hypothesis is valid for ergodic stationary statistical solutions. As the next step, we prove existence of such an ergodic solution and  investigate the ergodic structure of the set of all stationary statistical solutions.

For a shift invariant Borel set $\mathcal{A}\subset\mathcal{U}[\Ov{E}]$, we denote by $\mathcal{I}(\Ov{\mathcal{A}})$ the set of all invariant probability measures for the group $(S_{\tau})_{\tau\in R}$ on $\Ov{\mathcal{A}}$. According to Lemma~\ref{lem:ID1}, the set $\mathcal{I}(\Ov{\mathcal{A}})$ can be identified with the set of stationary statistical solutions in the sense of Definition~\ref{ID1}.

\begin{Theorem}\label{thm:AC1}
Let $\mathcal{A}\subset\mathcal{U}[\Ov{E}]$ be a shift invariant set of trajectories. 
\begin{enumerate}
\item There exists an ergodic stationary statistical solution in $\mathcal{I}(\Ov{A}).$
\item The law of every stationary statistical solution in $\mathcal{I}(\Ov{A})$ is the barycenter of a probability measure supported on the set $\mc I_e(\oline{\mc A})$
of ergodic stationary solutions in $\mathcal{I}(\Ov{\mathcal{A}})$.
\end{enumerate}
\end{Theorem}

\begin{proof}
By Theorem~\ref{IT1}, the set $\mathcal{I}(\Ov{\mathcal{A}})$ is non-empty. Moreover, it is easy to check that it is  convex and compact with respect to the weak convergence of probability measures. Hence according to Krein--Milman's theorem, there exists at least one extremal point of $\mathcal{I}(\Ov{\mathcal{A}})$. It is classical to show that the extremal points are exactly the ergodic invariant measures in $\mathcal{I}(\Ov{\mathcal{A}})$, see e.g. \cite[Proposition~12.4]{Phel}. So the first statement is proved. 

Let $\nu\in\mathcal{I}(\Ov{\mathcal A})$.  By Choquet's version of  Krein--Milman's theorem \cite[Section 3]{Phel}, $\nu$ is  the barycenter of a probability measure $m$ supported by the extremal points of $\mathcal{I}(\Ov{\mathcal A})$, i.e., by the ergodic measures on $\Ov{\mathcal A}$. Thus, we can write
\begin{equation}\label{eq:EE4}
\nu=\int_{\mathcal{I}_{e}(\Ov{\mathcal A})}\mu\,\D m(\mu),
\end{equation}
where $\mathcal{I}_{e}(\mathcal{\oline A})$ denotes the set of all ergodic measures on $\Ov{\mathcal A}$. 
\end{proof}

Alternatively to the second statement in Theorem~\ref{thm:AC1}, one may say that every invariant measure in $\mathcal{I}(\Ov{\mathcal A})$ is approximated by finite convex combinations of ergodic measures in $\mathcal{I}(\Ov{\mathcal A})$. We call \eqref{eq:EE4} the ergodic decomposition of $\nu$.

Recall that the pointwise ergodic theorem,  Theorem~\ref{AT1}, yields a.s. convergence of the ergodic averages for every stationary statistical solution. This is not fully satisfactory as the example at the beginning of this section shows. Indeed, in a certain setting, the  time independent solutions $\vr\equiv\vr_{0}>0$, $\vm\equiv 0$ generate invariant measures $\delta_{[\vr,\vm]}$ whose support is a singleton. Accordingly, Theorem~\ref{AT1} does not bring any interesting information. However, it can be observed that these trivial invariant measures are ergodic and further invariant measures are obtained by their convex combinations.

As a consequence of the ergodic decomposition, we obtain the following result on the structure of the  support of all the  invariant measures.

\begin{Corollary}\label{cor:EE3}
Let $\mathcal{A}\subset\mathcal{U}[\Ov{E}]$ be a shift invariant set of trajectories. Then
\begin{equation}\label{eq:EE5}
\cup\big\{\supp\nu;\,\nu\in\mathcal{I}(\Ov{\mathcal{A}})\big\}=\cup\big\{\supp\nu;\,\nu\in\mathcal{I}_{e}(\Ov{\mathcal{A}})\big\}.
\end{equation}
Moreover, two distinct ergodic invariant measures are singular.
\end{Corollary}

\begin{proof}
Since $\mathcal{I}_{e}(\Ov{\mathcal{A}})\subset \mathcal{I}(\Ov{\mathcal{A}})$, the right hand side of \eqref{eq:EE5} is a subset of the left hand side. To show the converse inclusion,
we recall that, by the ergodic decomposition \eqref{eq:EE4}, if a point $[\vr,\vm]$ belongs to the support of some invariant measure $\nu$ then it belongs to the support of at least one ergodic invariant measure, hence \eqref{eq:EE5} follows.

It remains to show that any two different ergodic invariant measures are singular.
Indeed, if there are two distinct invariant measures $\nu$, $\mu$, then there exists a Borel set $B\subset \Ov{\mathcal A}$ such that $\nu(B)\neq\mu (B)$.
According to Theorem~\ref{AT1} and Theorem~\ref{AC1} applied to the canonical process on $\mathcal{T}$ and $F={\mds 1}_{B}$, it follows that the ergodic averages converge $\nu$-a.s. to
$
\nu(B)
$
and $\mu$-a.s. to $\mu(B)$. If the two measures are not singular, there exists a Borel set $C\subset\Ov{\mathcal A}$ such that $\nu(C)>0$, $\mu(C)>0$ and we obtain a contradiction.
\end{proof}

\subsection{Minimality and ergodicity}

Recall that Theorem~\ref{IT1} can be applied on an arbitrary  shift invariant set $\mathcal{A}\subset\mathcal{U}[\Ov{E}]$. Due to Proposition~\ref{TP1}, the whole set $\mathcal{U}[\Ov{E}]$ is one possibility and in view of  Corollary~\ref{TP2}, another natural possibility is any $\omega$--limit set $\omega[\vr,\vm]$. Since all these examples are non-empty closed subsets of the compact set $\mathcal{U}[\Ov{E}]$, we may consider a partial ordering  given  by the relation $\subsetneq$. 
A non-empty, closed and shift invariant set $\mathcal{A}\subset\mathcal{U}[\Ov{E}]$ is called \emph{minimal} provided it does not contain any proper subset, which is also non-empty closed and shift invariant. It turns out that the minimal elements are intimately related to ergodic invariant measures.

\begin{Proposition}
The following statements hold true:
\begin{enumerate}
\item every non-empty closed and shift invariant set $\mathcal{A}\subset\mathcal{U}[\Ov{E}]$ contains a minimal set;
\item any two minimal sets are disjoint;
\item if $\mathcal{A}\subset\mathcal{U}[\Ov{E}]$ is minimal then there is at most one  invariant measure on $\mathcal{A}$ and it is ergodic.
\end{enumerate}
\end{Proposition}

\begin{proof}
The first point  is a consequence of Zorn's lemma, whereas the second one is an immediate consequence of minimality.
From minimality we know that if $E\subset \mathcal{A}$ is closed and shift invariant then either $E=\emptyset$ or $E=\mathcal{A}$. As a consequence,  every invariant measure on $\mathcal{A}$ is ergodic. If there were two different ergodic invariant measures, then their strict convex combination cannot be ergodic as it is not an extremal point. Hence we get a contradiction and there is only one invariant measure which is ergodic.
\end{proof}

The above lemma gives a recipe for the construction of ergodic invariant measures by restricting to minimal sets. Let us conclude this section with a result describing the structure of the minimal sets. The proof is not complicated and can be found in the literature on dynamical systems. To formulate the result, we recall that the set  $\{S_{\tau}[\vr,\vm];\tau\in R\}$ is called \emph{orbit} of $[\vr,\vm]\in\mathcal{U}[\Ov{E}]$.

\begin{Lemma}\label{lem:R1}
Let $\mathcal{A}\subset\mathcal{U}[\Ov{E}]$ be non-empty, closed and shift invariant. Then the following are equivalent:
\begin{enumerate}
\item $\mathcal{A}$ is minimal;
\item $\mathcal{A}$ is the orbit closure of every one of its points;
\item $\mathcal{A}$ is the $\omega$-limit of every one of its points.
\end{enumerate}
\end{Lemma}

\section{Validity of ergodic hypothesis}
\label{V}


To summarize our investigations, by Corollary~\ref{TP2} and Theorem~\ref{IT1}, every finite energy weak solution $[\vr,\vm]$ with a uniformly bounded energy converges (for a sequence of time shifts) to a solution $[r,\vc{w}]\in\omega[\vr,\vm]$ which belongs to the support of some invariant measure. According to Corollary~\ref{cor:EE3}, this solution $[r,\vc{w}]$  belongs to the support of at least one ergodic invariant measure and
Theorem~\ref{AC1}  provides the description of the limit of the corresponding ergodic averages.

On the other hand,  since also the set $\mathcal{U}[\Ov{E}]$ is compact, we could apply the Krylov--Bogolyubov procedure as in the proof of Theorem~\ref{IT1} directly to the family of measures
\begin{equation*}
T\mapsto\frac{1}{T}\int_{0}^{T}\delta_{S_{t}(\vr,\vm)}\D t.
\end{equation*}
In other words, for every finite energy weak solution $[\vr,\vm]$ with uniformly bounded energy, the ergodic averages converge --  up to a subsequence -- and define an invariant measure. The principle question of convergence of the whole sequence remains open in general. However, we may formulate the following ergodic localization principle, which permits to  reduce the problem of validity of the ergodic hypothesis to investigation of the structure of the $\omega$--limit sets.

\begin{Theorem}\label{thm:8.1}
Let $[\vr,\vm]$ be a finite energy weak solution with uniformly bounded energy and assume that there is a unique invariant measure $\nu$ on its $\omega$--limit set $\omega[\vr, \vm] \subset \mathcal{U}(\Ov{E})\subset\mathcal{T}$.
Then the ergodic hypothesis holds for $[\vr,\vm]$.

In particular, for every bounded continuous function $F:H\to R$,
one has
$$
\frac{1}{T}\int_{0}^{T}F(\vr(t,\cdot),\vm(t,\cdot))\,\D t \to \int_{H}F\,\D\nu_{0}\ \mbox{as}\ T\to\infty,
$$
where $\nu_{0}=\nu\circ \pi_{0}^{-1}$ is the  pushforward measure  on $H$ generated by the projection $\pi_{0}:\mathcal{T}\to H $  at time $t=0$.
\end{Theorem}

\begin{proof}
By  Krylov--Bogolyubov's argument, there is a sequence of times $T_{n}\to\infty$ and a measure $\nu\in\mathfrak{P}(\mathcal{U}[\Ov{E}])$ such that
\[ \frac{1}{T_{n}} \int_0^{T_{n}} \delta_{S_t (\varrho, \vm)} \,\D t \to\nu\ \mbox{narrowly in}\ \mathfrak{P}(\mathcal{U}[\Ov{E}]) . \]
Next, we claim that $\tmop{supp} \,\nu \subset \omega [\varrho, \vm]$.
If $[r, \vc{w}] \in \tmop{supp} \,\nu$, then for every $\varepsilon > 0$ it holds
$\nu (B_{\varepsilon} [r, \vc w]) > 0$, where $B_{\varepsilon} [r, \vc w]$ is the
open $\varepsilon$-neighborhood of $[r, \vc w]$ in $\mathcal{T}$. Since
$B_{\varepsilon} [r, \vc w]$ is open, the narrow convergence implies
\[ \liminf_{n \rightarrow \infty} \frac{1}{T_n} \int_0^{T_n} \delta_{S_t
   (\varrho, \vm)} (B_{\varepsilon} [r, \vc w]) \,\D t \geq \nu (B_{\varepsilon}
   [r, \vc w]) > 0, \]
hence
\[ \liminf_{n \rightarrow \infty} \int_0^{T_n} \delta_{S_t (\varrho, \vm)} (B_{\varepsilon} [r, \vc w]) \,\D t
   =\infty . \]
The integrand only takes values $0, 1$ and therefore there is an arbitrarily large time $\tau_{\varepsilon}>0$ so that
\[ [\varrho, \vm] (\cdot + \tau_\varepsilon) \in B_{\varepsilon} [r, \vc w] . \]
In other words, 
$[r, \vc w] \in \omega [\varrho, \vm]$.

Since there is at most one invariant measure on $\omega[\vr,\vm]$ by our assumption, $\nu$ is this measure and we get the convergence of the whole sequence
$$
\frac{1}{T}\int_{0}^{T}\delta_{S_{t}(\vr,\vm)}\D t\to \nu\ \mbox{narrowly in}\ \mathfrak{P}(\mathcal{U}[\Ov{E}]).
$$
Accordingly, for every bounded continuous function $G\in BC(\mathcal{T})$ it holds
\begin{equation}\label{eq:ERG}
\frac{1}{T}\int_{0}^{T}G(\vr(\cdot+t),\vm(\cdot+t))\,\D t \to \int_{\mathcal{T}}G(r,\vc{w})\,\D\nu(r,\vc{w})\ \mbox{as}\ T\to\infty.
\end{equation}
At this point, if $F\in BC(H)$ then taking $G=F\circ \pi_{0}$ in \eqref{eq:ERG} completes the proof.
\end{proof}

%
%
%

\def\cprime{$'$} \def\ocirc#1{\ifmmode\setbox0=\hbox{$#1$}\dimen0=\ht0
  \advance\dimen0 by1pt\rlap{\hbox to\wd0{\hss\raise\dimen0
  \hbox{\hskip.2em$\scriptscriptstyle\circ$}\hss}}#1\else {\accent"17 #1}\fi}


\end{document}